\documentclass[11pt,a4,leqno]{amsart}
\usepackage{amsmath,epsfig,graphicx,color}
\usepackage{comment}
\usepackage{ascmac}
\usepackage{mathrsfs}
\usepackage{setspace}
\usepackage{amsaddr}
\definecolor{darkblue}{rgb}{.2, 0.2,.8}
\definecolor{carageen}{rgb}{0,0.5,0.3}
\definecolor{darkred}{rgb}{1, 0,0}
\newcommand{\wlo}{without loss of generality}

\renewcommand{\a}{\alpha}
\renewcommand{\b}{\beta}
\newcommand{\red}{\color{darkred}}
\newcommand{\blue}{\color{darkblue}}
\newcommand{\green}{\color{carageen}}

\newcommand{\clt}{central limit theorem}

\newcommand{\garch}{{\rm GARCH}$(1,1)$}
\newcommand{\spr}{stochastic process}

\newcommand{\ex}{{\rm e}\,}

\newcommand{\asy}{asymptotic}

\newcommand{\ts}{time series}

\newtheorem{lemma}{Lemma}[section]

\newtheorem{theorem}[lemma]{Theorem}

\renewcommand{\P}{{\mathbb P}}
\newtheorem{proposition}[lemma]{Proposition}
\newtheorem{definition}[lemma]{Definition}
\newtheorem{corollary}[lemma]{Corollary}
\newtheorem{example}[lemma]{Example}
\newtheorem{exercise}[lemma]{Exercise}
\newtheorem{remark}[lemma]{Remark}

\newtheorem{tab}[lemma]{Table}

\newcommand{\MC}{Markov chain}

\newcommand{\bth}{\begin{theorem}}
\newcommand{\ethe}{\end{theorem}}
\newcommand{\sv}{stochastic volatility}
\newcommand{\bre}{\begin{remark}\em }
\newcommand{\ere}{\end{remark}}

\newcommand{\ble}{\begin{lemma}}
\newcommand{\ele}{\end{lemma}}
\newcommand{\sre}{stochastic recurrence equation}

\newcommand{\bde}{\begin{definition}}
\newcommand{\ede}{\end{definition}}
\newcommand{\bco}{\begin{corollary}}
\newcommand{\eco}{\end{corollary}}

\newcommand{\bpr}{\begin{proposition}}
\newcommand{\epr}{\end{proposition}}

\newcommand{\bexer}{\begin{exercise}}
\newcommand{\eexer}{\end{exercise}}

\newcommand{\bexam}{\begin{example}\rm }
\newcommand{\eexam}{\end{example}}

\newcommand{\efi}{\end{fig}}

\newcommand{\btab}{\begin{tab}}
\newcommand{\etab}{\end{tab}}

\newcommand{\rv}{random variable}

\newcommand{\sign}{{\rm sign}}

\newcommand{\var}{{\rm var}}

\newcommand{\rhs}{right-hand side}

\newcommand{\beao}{\begin{eqnarray*}}
\newcommand{\eeao}{\end{eqnarray*}\noindent}

\newcommand{\beam}{\begin{eqnarray}}
\newcommand{\eeam}{\end{eqnarray}\noindent}

\newcommand{\beqq}{\begin{equation}}
\newcommand{\eeqq}{\end{equation}\noindent}

\newcommand{\bce}{\begin{center}}
\newcommand{\ece}{\end{center}}

\newcommand{\barr}{\begin{array}}
\newcommand{\earr}{\end{array}}

\newcommand{\stp}{\stackrel{\P}{\rightarrow}}
\newcommand{\std}{\stackrel{d}{\rightarrow}}
\newcommand{\stas}{\stackrel{\rm a.s.}{\rightarrow}}

\newcommand{\stw}{\stackrel{w}{\rightarrow}}

\newcommand{\vague}{\stackrel{\lower0.2ex\hbox{$\scriptscriptstyle
                    \it{v} $}}{\rightarrow}}
\newcommand{\weak}{\stackrel{\lower0.2ex\hbox{$\scriptscriptstyle
                    \it{w} $}}{\rightarrow}}
\newcommand{\what}{\stackrel{\lower0.2ex\hbox{$\scriptscriptstyle
                    \it{\hat{w}} $}}{\rightarrow}}

\newcommand{\bdis}{\begin{displaymath}}
\newcommand{\edis}{\end{displaymath}\noindent}

\newcommand{\R}{\mathbb{R}}

\newcommand{\nto}{n\to\infty}
\newcommand{\mto}{m\to\infty}

\newcommand{\xto}{x\to\infty}

\newcommand{\ov}{\overline}
\newcommand{\wt}{\widetilde}

\newcommand{\vep}{\varepsilon}

\newcommand{\la}{\lambda}

\newcommand{\regvary}{regularly varying}
\newcommand{\slvary}{slowly varying}
\newcommand{\regvar}{regular variation}

\newcommand{\bbr}{{\mathbb R}}

\newcommand{\bbz}{{\mathbb Z}}
\newcommand{\Z}{{\mathbb Z}}

\newcommand{\con}{convergence}

\newcommand{\evt}{extreme value theory}

\newcommand{\st}{such that}
\newcommand{\fif}{if and only if}
\newcommand{\wrt}{with respect to}
\newcommand{\chf}{characteristic function}
\newcommand{\fct}{function}

\newcommand{\ds}{distribution}

\newcommand{\rep}{representation}

\newcommand{\seq}{sequence}

\newcommand{\mgf}{moment generating function}
\newcommand{\ld}{large deviation}

\newcommand{\E }{{\mathbb E}}
\renewcommand{\P }{{\mathbb P}}

\newcommand{\1}{{\mathbf 1}}

\allowdisplaybreaks

\begin{document}
\onehalfspacing
\today
\bibliographystyle{plain}
\title[The Gaussian CLT for a stationary time
series  with infinite variance]{
The Gaussian central limit theorem for a stationary time series with infinite variance}
\thanks{Muneya Matsui's research is partly supported by the JSPS Grant-in-Aid for Scientific Research C
(19K11868).} 

%\author[M. Matsui]{Muneya Matsui}
%\address{Department of Business Administration, Nanzan University, 18
%Yamazato-cho, Showa-ku, Nagoya 466-8673, Japan.}
%\email{mmuneya@gmail.com}
%
%\author[T. Mikosch]{Thomas Mikosch}
%\address{Department  of Mathematics,
%University of Copenhagen,
%Universitetsparken 5,
%DK-2100 Copenhagen,
%Denmark}
\author[Matsui, Mikosch]{Muneya Matsui, Thomas Mikosch}
\address{M. Matsui\\ 
Department of Business Administration, 
Nanzan University\\
18 Yamazato-cho Showa-ku Nagoya, 466-8673, Japan}
\email{mmuneya@gmail.com}
\address{T. Mikosch\\
Department  of Mathematics,
University of Copenhagen \\
Universitetsparken 5,
DK-2100 Copenhagen,
Denmark}
\email{mikosch@math.ku.dk}
%\address{O. Wintenberger\\
%Laboratoire de Probabilit\'es, Statistique et Mod\'elisation\\
%Sorbonne Universit\'e, UPMC Université Paris 06, F-75005, Paris, France\\
%and\\
%Institut CNRS Pauli, Vienna University.}
%\email{olivier.wintenberger@sorbonne-universite.fr}

\begin{abstract}
We consider a borderline case: the \clt\ for a strictly stationary \ts\
with infinite variance but a Gaussian limit. In the iid
case a well-known sufficient condition for this \clt\
is \regvar\ of the marginal \ds\ with tail index $\a=2$. 
In the dependent case we assume the stronger condition of sequential 
\regvar\ of the 
\ts\ with tail index $\a=2$. We assume that a sample of size $n$
from this \ts\ can be split into $k_n$ blocks of size $r_n\to\infty$ \st\
$r_n/n\to 0$ as $\nto$ and that the block sums are \asy ally independent.
Then we apply classical central limit theory for row-wise iid 
triangular arrays. The necessary and sufficient conditions for such 
independent block sums will be verified by using \ld\ results for the \ts .  
We derive  the \clt\ for $m$-dependent \seq s, linear processes, \sv\ processes and 
solutions to affine \sre s whose marginal \ds s  have infinite variance and are
\regvary\ with tail index $\a=2$.
\end{abstract}
%\keywords{Regularly varying \seq , \clt}
%\subjclass{Primary 60F05; Secondary 60E07, 60E10, 60G70, 62E20}
\maketitle
\begin{small}
\noindent
{\em MSC2020 subject classifications:}\,{\rm Primary 60F05;\, Secondary 60E10 60G70 62E20}
\end{small} %\\

%{\em Some personal words from TM.} Dear Peter! I loved your humor and intellectual spirit from the 
%first time we met -- at the famous \ts\ meeting in Heidelberg in 1992. 
%I have learned \tsa\ from your book with Richard. 
%I had the privilege of living in your house in Fort Collins and drinking your lime liquor 
%in your absence. We biked along the coast in Copenhagen and enjoyed  Japanese cuisine in Fort Collins. Thank you for every minute we spent together.

\section{Introduction}
\subsection{The \clt\ for iid \seq s}
We start with an iid regularly varying (strictly) stationary sequence $(X_t)_{t\in\bbz}$
with generic element $X$ and tail index $\alpha>0$. This means, in particular,
that there exist a \slvary\ \fct\ $L(x)$ and constants $p_\pm\ge 0$ \st
\beam\label{eq:2}
\left\{\barr{cccc}
\P(|X|>x)&=&\dfrac{L(x)}{x^\a}\,,&\qquad x>0\,,\\
\dfrac{\P(\pm X>x)}{\P(|X|>x)}&\to &p_\pm\,, &\qquad \xto\,.
\earr
\right.
\eeam  
\par 
Consider the partial sum process
\beao
S_n=X_1+\cdots +X_n\,,\qquad  n\ge 1\,.
\eeao
We choose normalizing constants as the unique solutions to the equations
\beam\label{eq:3}
n\,\P(|X|>a_n)+ \dfrac{n}{a_n^2} \E[X^2\ \1(|X|\le a_n)]=1\,.
\eeam 
Classical central limit theory yields the following results in the iid case;
see \cite{pruitt:1981}, \\
\cite{gnedenko:kolmogorov:1954}, \cite{petrov:1995}, \cite{ibragimov:linnik:1971}, \cite{feller:1971}:\\[2mm]
{\bf 1.} {\em The case $\a<2$.} Then
\beam\label{eq:1}
\dfrac{S_n-b_n}{a_n}\std \xi_\a\,,\qquad \nto\,,
\eeam  
where $\xi_\a$ has an $\a$-stable \ds , the centering constants can be chosen \st\ 
$b_n=n\,\E[X\,\1(|X|\le a_n)]$. Moreover, $b_n=n\,\E[X]$ for 
$\a\in (1,2)$ and $b_n=0$ for $\a\in (0,1)$ are possible centering constants.\\[1mm]{\bf 2.} {\em The finite-variance case.}
The $\a$-stable \clt\ \eqref{eq:1}
is supplemented
by the classical \clt : if $\sigma^2=\var(X)\in (0,\infty)$, $\E[X]=0$, then \eqref{eq:1}
holds with $b_n=0$, $a_n=\sigma\,\sqrt{n}$ and a standard normal
limit $\xi_2$. The tail behavior of $X$ is not relevant.
We can choose $a_n$ according to \eqref{eq:3} while 
$n\,\P(|X|>a_n)=o(1)$ as $\nto$. The finite-variance case includes certain 
\ds s satisfying 
\eqref{eq:2} for $\a=2$.\\[1mm]
{\bf 3.} {\em The case $\a=2$, $\var(X)=\infty$.} Assume
\eqref{eq:2} and  \wlo\ that $\E[X]=0$.
Then \eqref{eq:3} turns into
\beam\label{eq:3a}
\dfrac{n}{a_n^2}\,L(a_n)+ \dfrac{n}{a_n^2}\,\E[X^2\,\1(|X|\le a_n )]=1\,,
\eeam
and $K(x)=\E[X^2\,\1(|X|\le x)]$, $x>0$, is a 
\slvary\ \fct\ increasing to infinity, and $\P(|X|>x)=o(x^{-2}K(x))$ as $\xto$. Hence we can choose
$a_n= \sqrt{n}\,\ell(n):=\sqrt{n\,K(a_n)}$ where $\ell(x)\to\infty$ as $\xto$ is a \slvary\ \fct . Moreover, the \clt\ with normal limit $\xi_2$ remains valid with the 
unconventional normalizing constants $(a_n)$. As a matter of fact, the necessary
and sufficient condition for the \clt\ with normalization \eqref{eq:3a} is 
slow variation  of $K(x)$. It is implied by \regvar\ of $X$ with index $\a=2$ in the sense of 
\eqref{eq:2} but slow variation of $K(x)$ is slightly more general. In Example~\ref{exam:1} we consider an iid \seq\ $(X_t)$ with marginal tail $\P(X>x)=x^{-2}N(x)$, 
non-\slvary\  $N(x)$ \st\ $\var(X)=\infty$, $K(x)$ is \slvary\  and 
the Gaussian \clt\ holds. 

\subsection{Goals of the paper}
We will study the \clt\ with Gaussian limit  
when $(X_t)$ constitutes a real-valued  {\em \regvary } (strictly) stationary \ts\ with 
index $\a=2$ and $\var(X)=\infty$. 
This means that a generic element $X$ satisfies 
\eqref{eq:2}, and there exist a Pareto$(\a)$ \rv\ $Y_\a$ with tail
$\P(Y_\a>x)=x^{-\a}$, $x>1$, and a  \seq\ of \rv s $(\Theta_t)$ which is independent
of $Y_\a$
\st\ for every $h\ge 0$,
\beam \label{eq:4}
\P\big(x^{-1}( X_{-h},\ldots,X_h)\in \cdot\, \big|\,|X_0|>x \big)
\stw \big(Y_\a\,(\Theta_{-h},\ldots,\Theta_h)\in \cdot\big)\,,\qquad \xto\,.
\eeam
The process $(\Theta_t)_{t\in\bbz}$ is the {\em spectral tail process} of the 
\regvary\ \seq\ $(X_t)$.
\par
The notion of a \regvary\ stationary \ts\ (for general $\a>0$) was introduced in different form 
by \cite{davis:hsing:1995}, and in the present form \eqref{eq:4} by \cite{basrak:segers:2009}. For recent book treatments, see  
\cite{kulik:soulier:2020}, \cite{mikosch:wintenberger:2024}. These references also contain chapters about $\a$-stable central limit theory
in the infinite variance case when $\a<2$, and about \evt\ for general $\a>0$.
\par
Under the conditions $\a=2$ and $\var(X)=\infty$
central limit theory for \ts\ with a Gaussian limit has been treated 
in the literature quite rarely. We are aware of work by \\
\cite{jakubowski:swewczak:2020} who proved the \clt\ for $\a=2$ 
for a\\ \garch\ process. The \regvar\ properties of such processes are well known; see \\ \cite{buraczewski:damek:mikosch:2016} for 
a book treatment. \cite{peligrad:2013}, \cite{peligrad:2022} proved the \clt\ for general 
linear processes with iid and martingale difference noise, 
including the infinite variance case corresponding 
to \regvar\ with index $\a=2$. In this case, the marginal tail of $|X|$ of the form 
$\P(|X|>x)\sim c\,x^{-2}L(x)$, $\xto$, is inherited from the \regvary\ tail of the 
underlying noise; see \cite{mikosch:samorodnitsky:2000}
for best possible conditions under which these results hold 
in the case of iid noise. \cite{jakubowski:swewczak:2020} 
proved the \clt\ for a GARCH(1,1) process by an application
of the \clt\ for martingale differences. \cite{peligrad:2013}, \cite{peligrad:2022}
exploited the structure of the linear process and also used martingale central
limit theory to derive their results. The structure of linear processes with iid 
\regvary\ noise with index $\a<2$ was also exploited by \cite{davis:resnick:1985,davis:resnick:1985a} who proved the 
$\a$-stable \clt\ under more restrictive conditions 
than in \\ \cite{peligrad:2013,peligrad:2022}. 
\cite{philipps:solo:1992} provided 
some general techniques for dealing with \asy\ theory for linear processes. 
\par
In our paper we follow a different path of proof. Assuming mixing conditions,
we split the sample $X_1,\ldots,X_n$ into $k_n\to \infty$ \asy ally independent
blocks of size $r_n=n/k_n\to\infty$. This reduces the \clt\ for $(S_n)$ to a \clt\
for $k_n$ iid block sums of the size $r_n$. Then we can apply 
classical limit theory for row-wise iid triangular arrays. The necessary
and sufficient conditions for these results require 
uniform \ld\ approximations of the 
type $\P(S_{r_n}>a_n\,y)\approx r_n\,\P(X>a_n\,y)$ where $(a_n)$ is defined in \eqref{eq:3}, \eqref{eq:3a}, and $y>0$ is chosen from a suitable region. We refer to 
Theorem~\ref{clt:thm:petrov} below for the necessary and sufficient conditions 
for the \clt\ for row-wise iid triangular arrays. 
Using precise \ld\ bounds in these conditions, we can 
derive the \clt\ $(S_n-\E[X])/a_n\std N(0,\sigma^2)$ in the boundary case $\a=2$ and 
determine the value $\sigma^2$.
\par
The paper is organized as follows. In Section~\ref{sec:prelim}
we introduce the basic mixing condition \eqref{eq:5}
that ensures the \asy\ independence
 of the block sums based on the increasing sample. This condition
is assumed throughout the paper. It is expressed in terms of \chf s and can easily
be verified under classical mixing conditions such as $m$-dependence, strong mixing or $\beta$-mixing with appropriate rates.
\par The basic \clt\ for independent block sums is provided in Theorem~\ref{clt:thm:petrov}. Its conditions require  uniform bounds for the tails of the block sums,
partly hidden in the truncated moments of these sums. The consequent
use of \regvar\ for \ts\ and related \ld\ techniques are 
perhaps the main differences to the literature on the \clt\ 
for \regvary\ \ts\ with tail index $\a=2$;
see \cite{jakubowski:swewczak:2020,peligrad:2013,peligrad:2022} and 
the preceding discussion. An advantage of our approach is its generality: it does not
depend on a particular structure of the \ts . Moreover, under the mixing condition
\eqref{eq:5} the conditions of  Theorem~\ref{clt:thm:petrov} are necessary and sufficient for the \clt\ to hold. The aforementioned literature \cite{jakubowski:swewczak:2020,peligrad:2013,peligrad:2022} depends on the martingale \clt\  when $\a=2$, $\var(X)=\infty$; cf. Theorem 8 in Chapter V of 
\cite{liptser:shiryaev:1986}. Writing $V_n=\sum_{t=1}^nX_t^2$ and assuming $\E[X]=0$, the latter approach also allows
one to derive the joint \con\ of $a_n^{-1}(S_n-b_n,V_n)\std (\xi_2,1)$ and the 
corresponding studentized \clt\ $(S_n-b_n)/V_n\std \xi_2$ as $\nto$. 
\par
In Section~\ref{sec:prelim} we discuss how some of the necessary
conditions for the \clt\ of independent block sums can be verified.
The main difficulty of Theorem~\ref{clt:thm:petrov} is the control
of the \asy\ variance $\sigma^2$ in the \clt : it appears as a scaled  
limit of the 
variance of truncated block sums. In this context the \ld\ results show its  full power.
In Section~\ref{lem:x2} we provide the form of the \asy\ variance in terms 
of the spectral tail process $(\Theta_t)$ for an $m$-dependent process $(X_t)$. 
This is due to a \ld\ result proved in \cite{mikosch:wintenberger:2016}. It depends on the crucial condition 
that $\P(\Theta_{-m}=\cdots=\Theta_{-1}=0)>0$. The derivation of the backward 
spectral tail process is, in general, a hard problem but it is simple
for $m$-dependent linear processes and \sv\ models; see Section~\ref{exam:linear} 
and Example~\ref{exam:sv}, respectively. Regularly varying linear and \sv\ models with 
infinite memory can be approximated by \regvary\ $m$-dependent versions of these
processes; see Sections~\ref{exam:linear} and \ref{sec:svmodel}.  
These approximations work because the \regvar\ of the aforementioned 
processes is caused by their underlying iid \regvary\ noise. 
\par
This is in sharp contrast to the \regvary\ solution of an affine \sre ; see
 \cite{buraczewski:damek:mikosch:2016} for the theory of such processes. Truncations of the infinite series \rep s of these solutions
are not \regvary , in general, and may have all moments. Therefore $m$-dependent approximations to the infinite series do not
work in general. In Sections~\ref{sec:sre1} and \ref{subsec:sre2} we deal with the \clt\ for $\a=2$ by applying precise
\ld\ results obtained in \cite{buraczewski:damek:mikosch:zienkiewicz:2013} in the so-called Kesten-Goldie case and in \cite{konstantinides:mikosch:2005}
in the so-called Grincevi\v cius-Grey case.  These results allow us to determine the \asy\ variance in the corresponding \clt s via Theorem~\ref{clt:thm:petrov}. 
\par
The  proofs of  \clt s in the borderline case $\a=2$ show
that both \evt\ and \ld\ results get into a marriage. For $\a>2$ one can work with the autocorrelation \fct\ of the underlying \ts\ and derive central limit theory with Gaussian
limits. This is in stark contrast with $\a<2$ where \ld s and extremes are
at the heart of the problem for deriving $\a$-stable limit theory; see \cite{mikosch:wintenberger:2013,mikosch:wintenberger:2014,mikosch:wintenberger:2016} and \cite{matsui:mikosch:wintenberger:2024a,matsui:mikosch:wintenberger:2024b}.

\section{Preliminaries}\label{sec:prelim}\setcounter{equation}{0}
In what follows, we consider a real-valued \regvary\ strictly stationary \seq\ $(X_t)$
(in the sense of \eqref{eq:4}) with 
generic element $X$ and index $\a=2$); we will always refer to a {\em stationary} \seq .
We also assume that $X$ has mean zero and 
infinite variance. For the 
definition of sequential \regvar\ we refer to the previous section. 
For ease of presentation, we assume that the block number $k_n=n/r_n$ 
corresponding to the block length $r_n$ is an integer. The cases $n/r_n\in ([k_n], [k_n]+1)$ only lead to small changes in the proofs. 
For any \rv\ $Y$ we write $\varphi_Y(u)=\E[\ex^{i\,u\,Y}]$, $u\in\bbr$, 
for the \chf\ of $Y$.
We assume that a mixing condition in terms of \chf s holds: for every $u\in\bbr$,
\beam\label{eq:5}
\varphi_{S_n/a_n}(u)=\big(\varphi_{S_{r_n}/a_n}(u)\big)^{k_n}+o(1)\,,\qquad \nto\,.
\eeam   
This relation holds under classical conditions 
such as $m$-dependence, strong mixing, assuming suitable 
decay rates of the  mixing coefficients in the latter case; see Lemma~9.1.3
in \\ \cite{mikosch:wintenberger:2024} 
for sufficient conditions for \eqref{eq:5} and Proposition~9.1.6 for coupling 
arguments leading to this condition.
In the $m$-dependent case one can choose $r_n\to\infty$ arbitrarily slowly.
If $(X_t)$ is strongly mixing with mixing coefficients $(\alpha_h)_{h\ge 0}$
then  $\eqref{eq:5}$ is satisfied if an anti-clustering condition holds 
((9.1.7) in \cite{mikosch:wintenberger:2024}) and
one can find \seq s $(r_n)$ and $(\ell_n)$
\st\ $\ell_n\to\infty$, $\ell_n/r_n\to0$ and $k_n\,\a_{\ell_n}\to 0$; see 
Remark~9.1.4 in \cite{mikosch:wintenberger:2024}.

\bexam\label{exam:mixing}
Assume $\a_h\le c\,\rho^h$ for some $c>0$ and $\rho\in (0,1)$. Then we can choose 
$\ell_n=[C\,\log n]$ for some $C>0$ \st\ 
$k_n\,\rho^{\ell_n}\sim k_n\,\rho^{C\,\log\,n}\to 0$. In particular, we can choose $r_n=[(\log n)^{1+\vep}]$ for 
arbitrarily small $\vep>0$, and then both $k_n \alpha_{\ell_n}\to 0$ and $\ell_n/r_n \to 0$ are satisfied. This is in contrast to the examples treated in Sections \ref{sec:sre1}--\ref{subsec:sre2}. Then the magnitude of $r_n$ depends very much on the \slvary\ \fct\ $L$ in $\P(X>x)=L(x)\,x^{-2}$.  
%Write $\log_k$
%for the $k$ times iterated logarithm. 
%Then we can choose $\ell_n=[C\,\log_kn]$ for  some $C>0$ and $k\ge 2$ \st\ 
%$r_n\,\rho^{\ell_n}\sim r_n\,(\log_{k-1} n)^{C\,\log\,\rho}\to 0$.   {\red In particular, we can choose $r_n=[\log_{k-1}n]$, $r_n\,\a_{\ell_n}\to 0$ and $\ell_n/r_n\to0$.}
\eexam
Under \eqref{eq:5}, the \clt\ for $(S_n)$ with normalization $(a_n)$ 
holds \fif\ it does
for $(T_n/a_n):=((S_{r_n,1}+\cdots +S_{r_n,k_n})/a_n)$ where $(S_{r_n,i})_{i=1,\ldots,k_n}$
are iid copies of $S_{r_n}$; here and in what follows, 
we always assume that the last (incomplete) block sum can be neglected, i.e., $(S_n-T_n)/a_n\stp 0$ as $\nto$. 
\bth\label{clt:thm:petrov}\cite[Section 4.1, Theorem 4.1]{petrov:1995} 
%{\red follows from CLT?} Assume the infinite smallness condition $S_{r_n}/a_n\stp 0$.
The \clt\ $T_n/a_n\std N(0,\sigma^2)$ holds for some $\sigma^2>0$ 
\fif\ the following three conditions are satisfied: for every $\vep>0$, as $\nto$,
\beam
k_n \,\P\big(|S_{r_n}/a_n|>1\big) &\to& 0\,,\label{eq:6a}\\
k_n\,\var\big((S_{r_n}/a_n)\,\1(|S_{r_n}/a_n|\le \vep) \big)&\to& \sigma^2\,,\label{eq:6b}\\
k_n \,\E\big[(S_{r_n}/a_n)\,\1(|S_{r_n}/a_n|\le 1) \big]&\to& 0\,.\label{eq:6c}
\eeam
\ethe
\bre\label{rem:1}
Condition \eqref{eq:6a} will follow from \ld\ results of the type
$\P(|S_{r_n}/a_n|>1)\sim c\,r_n\,\P(|X|>a_n)$ for some positive constant $c$ and 
the fact that $n\,\P(|X|>a_n)\to 0$ in the case $\a=2$; see \eqref{eq:7} and comment (2) below.  Conditions \eqref{eq:6a} and 
\eqref{eq:6b} are responsible for the appearance of a Gaussian limit \ds . In particular,
\eqref{eq:6b} is insensitive \wrt\ the value of $\vep$. In the case of an $\a$-stable
limit with $\a<2$, the limit in \eqref{eq:6b} should be proportional to $\vep^{-\a}$.
\ere
In what follows, any positive constant whose value is not of interest will be 
denoted by $c$.
The next result yields some simple sufficient conditions for \eqref{eq:6c} to hold.
\ble\label{lem:6c} 
Assume the aforementioned conditions on the \regvary\ stationary \seq\ $(X_t)$
with index $\a=2$, in particular \eqref{eq:5}, $\E[X]=0$ and $\var(X)=\infty$. Choose $a_n=\sqrt{n}\ell(n)=\sqrt{n\,K(a_n)}$. 
Also assume that one of the 
following two conditions holds.
\begin{enumerate}
\item[\bf 1.]
$S_n$ is symmetric for every $n\ge 1$.
\item[\bf 2.] The following \ld\ result holds for some 
non-negative constant $c_0$ and a \seq\ $s_n>a_n$:
\beam\label{eq:7}
\sup_{y\in [a_n,s_n]}\Big|\dfrac{\P(|S_{r_n}|>y)}{r_n\,\P(|X|>y)}- c_0\Big|\to 0\,,\qquad
\nto\,,
\eeam
and for an arbitrarily small $\delta>0$ there exist positive constants $\gamma$
and $c$ \st\ 
the moment inequality 
\beam\label{eq:8}
\E[|S_n|^{2-\delta}]\le c\,n^{\gamma}\,,\qquad n\ge 1\,,
\eeam
holds. 
Moreover, assume
\beam\label{eq:8a}
\dfrac{n^{1/2}\,r_n^{\gamma-1}}{\ell(n)\,s_n^{1-\delta}}\to 0\,,\qquad \nto\,.
\eeam
\end{enumerate}
Then the relation \eqref{eq:6c} holds.
\ele
\subsection*{Some comments}
\begin{enumerate}
\item
Condition {\bf 1.}  is satisfied for a stationary \sv\ model given by
$X_t=\sigma_t\,Z_t$, $t\in\bbz$, where $(\sigma_t)$ is stationary, 
$(Z_t)$ is iid
symmetric,  $\sigma_t=f(Z_{t-1},Z_{t-2},\ldots)$ for 
some non-negative measurable $f$. 
Then $S_{r_n}$ is symmetric and $\sign(S_{r_n})$, $|S_{r_n}|$ are independent.
This case includes GARCH processes driven by standard normal or $t$-distributed 
noise; see \cite{buraczewski:damek:mikosch:2016} and \cite{mikosch:wintenberger:2024} for further reading on
\sv\ and GARCH models and their properties.
Alternatively, $S_{r_n}$ is symmetric if 
$(\sigma_t)$ and $(Z_t)$ are independent and a generic element
$Z$ is symmetric.
\item
Large deviations of the type \eqref{eq:7} 
can be found in \cite{mikosch:wintenberger:2013,mikosch:wintenberger:2014,mikosch:wintenberger:2016}
for general \regvary\ stationary \seq s, in  
\cite{buraczewski:damek:mikosch:zienkiewicz:2013,buraczewski:damek:mikosch:2016}
and \cite{konstantinides:mikosch:2005} 
for solutions to affine \sre s, and in \cite{mikosch:samorodnitsky:2000} 
for linear processes. These results are in the same spirit as derived for iid \regvary\ \seq s by 
\cite{nagaev:1969a,nagaev:1969,nagaev:1977,nagaev:1979}; see also Appendix F in \cite{mikosch:wintenberger:2024} for a
collection of such results in the iid case.
\item
Inequalities of the type \eqref{eq:8} are known for martingale difference
\seq s $(X_t)$ via the Burkholder-Davis-Gundy inequality
(see \cite{williams:1991}, Section~14.18):
\beao
\E[|S_n|^{2-\delta}]\le c\, \E\Big[\Big(\sum_{i=1}^n X_i^2\Big)^{1-\delta/2}\Big]
 \le c\,n\,\E[|X|^{2-\delta}]\,, 
\eeao
where we used the triangle inequality in $L^{1-\delta/2}$ in the last step.
This inequality remains valid in the situations described in comment (1).
%{\blue in (1) the \rv s are symmetric, hence $X_t=|X_t|\vep_t$, $(|X_t|)$ independent of
%iid Rademacher \rv s $(\vep_t)$, }
%\beao\blue
%\E[|S_n|^{2-\delta}]=\E\Big[\Big(\sum |X_t|\vep_t\Big)^{2-\delta}\Big| (|X_t|)\Big]
%\eeao
Indeed,
conditioning on $(|X_t|)$, one can apply Khintchine's inequality for sums of weighted iid Rademacher \seq s; see \cite{petrov:1995}, 2.6.18 on p. 82.
A worst case inequality is the triangular 
%{\blue
%\beao
%(\E[|S_n|^{2-\delta}])^{1/(2-\delta)}\le n\,(\E[|X|^{2-\delta}])^{1/(2-\delta)}
%\eeao} 
one: for small $\delta>0$,
\beao
\E[|S_n|^{2-\delta}]
 \le n^{2-\delta}\,\E[|X|^{2-\delta}]\,. 
\eeao
\item
Condition \eqref{eq:8a} is satisfied if $(s_n)$ can be chosen in such a way
that $s_n/n^v\to\infty$ for an arbitrarily large $v>0$. This choice is always possible in the models
mentioned in comment~(2).
\end{enumerate}
\begin{proof} {\bf 1.} The statement is trivial in this case.\\[2mm]
{\bf 2.} Since $\E[S_{r_n}]=0$ and \eqref{eq:7} holds we have 
\beao
&& k_n\, \big|\E\big[(S_{r_n}/a_n)\,\1(|S_{r_n}/a_n|\le 1)\big]\big| \\
&&\le k_n\,\big|\E\big[|S_{r_n}/a_n|\,\1(|S_{r_n}/a_n|> 1)\big]\big|\\
&&=\dfrac{k_n}{a_n}\Big(\int_{a_n}^{s_n}+\int_{s_n}^\infty\Big) 
\P\big(|S_{r_n}|>y\big)\,dy + k_n \P\big(|S_{r_n}|>a_n \big)\\
&&\le  c\,\dfrac{k_n}{a_n} r_n\,\int_{a_n}^{s_n}\P(|X|>y)\,dy
+ \dfrac{k_n}{a_n} \int_{s_n}^\infty \P\big(|S_{r_n}|>y\big)\,dy + k_n \P\big(|S_{r_n}|>a_n \big)\\
&&=:J_1+J_2 + J_3\,.
\eeao
\begin{comment}
%{\blue I think $J_3$ is not there.}
\textcolor{blue}{The reason for $J_3$, please see the proof of Lemma \ref{lem:ldsre} is as follows. 
\begin{align*}
&\int_0^\infty \P(|X| \1(|X|>a_n)>y)dy \\
&= \int_0^{a_n} \P(|X| \1(|X|>a_n)>y)dy + \int_{a_n}^\infty \P(|X| \1(|X|>a_n)>y)dy \\
&= a_n \P(|X|>a_n) + \int_{a_n}^\infty \P(|X|>y)dy. 
\end{align*}
\end{comment}
The fact $J_3=o(1)$ follows directly from \eqref{eq:7}.
By Karamata's integral theorem (see\\ \cite{bingham:goldie:teugels:1987}) and the definition of $(a_n)$, $J_1=O(n\,\P(|X|>a_n))=o(1)$. 
For $J_2$ we apply Markov's inequality of
the order $2-\delta$ for an arbitrarily small $\delta>0$:
\beao
J_2&\le& c\,\dfrac{k_n}{a_n}\,  \E[|S_{r_n}|^{2-\delta}]\,
\int_{s_n}^\infty y^{-2+\delta}\,dy\\
&\le &c\, \dfrac{n\,\E[|S_{r_n}|^{2-\delta}]}{r_n\,a_n\,s_n^{1-\delta}}= O\Big( \dfrac{n^{1/2}\,r_n^{\gamma-1}}{\ell(n)\,s_n^{1-\delta}}\Big)\,.
\eeao
In the last step we used the structure of $a_n=\sqrt{n}\,\ell(n)$ and the bound
$\E[|S_{r_n}|^{2-\delta}]\le c\,r_n^\gamma$ for some positive $\gamma$;
see \eqref{eq:8}. The \rhs\ converges to zero by virtue of \eqref{eq:8a}.\\[2mm]
\begin{comment}
{\bf Unsuccessful}
Since $\E[X]=0$ we have $\E[S_{r_n}]=0$. Hence
\beao\lefteqn{
k_n\,\Big|\E\Big[(S_{r_n}/a_n)\,\1(|S_{r_n}/a_n|\le 1) \Big]\Big|}\\
&\le & n \,\Big|\E\big[(X/a_n)\,\1\big(|S_{r_n}/a_n|\le 1\,,|X/a_n|> 1 \big)\big]\Big|
+ n \,\Big|\E\big[(X/a_n)\,\1\big(|S_{r_n}/a_n|\le 1\,,|X/a_n|\le  1 \big)\big]\Big|\\
%&\le &
%n \,\Big|\E\big[|X/a_n|\,\1\big(|X/a_n|> 1 \big)\big]\\
&\le &\dfrac{\E\big[|X/a_n)|\,\1(|X|> a_n)\big]}{\P(|X|>a_n)}
\,[n\,\P(|X|>a_n)]
+ n \,\Big|\E\big[(X/a_n)\,\1\big(|S_{r_n}/a_n|> 1\mbox{ or } |X/a_n|> 1 \big)\big]\Big|\eeao
The first term converges to zero by an application of Karamata's theorem and 
since $n\,\P(|X|>a_n)=o(1)$. The second term is bounded by
\beao
&\le & n \,\Big|\E\big[|X/a_n|\,\1\big(|X|>a_n\big)\big]\Big|+ 
n\,\Big|\E\big[(X/a_n)\,\1\big(|S_{r_n}/a_n|> 1\,, |X/a_n|\le  1 \big)\big]\Big|
\eeao
The first term is again negligible. The second one is bounded by
\beao
\le k_n \Big|\sum_{i=1}^{r_n}\E\big[(X_i/a_n)\1\big(|S_{r_n}/a_n|> 1\,, |X_i/a_n|\le  1 \big)\big]\Big|
\eeao 
\end{comment}
\end{proof}

\section{Examples}\setcounter{equation}{0}
While the conditions 
\eqref{eq:6a} and \eqref{eq:6c} are relatively simple to verify,  
\eqref{eq:6b} calls for more complex arguments. In this section we will
consider several examples where we can show \eqref{eq:6b}.
\subsection{The $m$-dependent case}
We consider an $m$-dependent stationary \seq\ $(X_t)$. In this case, the mixing 
condition \eqref{eq:5} is trivially satisfied and one can choose $r_n\to\infty$
arbitrarily slowly.
 Recall the definition of the spectral
tail process $(\Theta_t)$ for a \regvary\ \seq\ $(X_t)$; see \eqref{eq:4}.
By its definition, the spectral tail process for an $m$-dependent process degenerates, i.e.,
$\Theta_t=0$ for $|t|>m$. 
\ble\label{lem:x2}
Assume $m$-dependence for a
 \regvary\ stationary \seq\ $(X_t)$ with index $\a=2$, infinite variance and mean zero.
Then \eqref{eq:6b} holds with  
\beam\label{eq:sigmasqare}
\sigma^2=\E\Big[\Big( \sum_{ t=0}^m \Theta_t\Big)^2\,\1\big(\Theta_{-i}=0,i=1,\ldots,m\big)\Big].
\eeam
%{\blue if the indicator is zero ithe \rhs\ is zero, that's ok}
\ele
\begin{proof} Due to $m$-dependence we can choose the block size $r_n\to\infty$ 
arbitrarily small. 
We have for every $\vep>0$ and arbitrarily small $\delta>0$,
\begin{align*}
& k_n\, \E\big[(S_{r_n}/a_n)^2\1\big(|S_{r_n}|\le \vep a_n\big)\big] \\
%&=&\dfrac{k_n}{a_n^2}\int_0^{(\vep a_n)^2}\P\big(|S_{r_n}|>\sqrt{y}\big)\,dy\\
&= \dfrac{k_n}{a_n^2}\Big(\int_{0}^{r_n^{1+\delta}}+\int_{r_n^{1+\delta}}^{(\vep a_n)^2}\Big)\P\big(|S_{r_n}|>\sqrt{y}\big)\,dy 
- k_n \vep^2\ \P\big( |S_{r_n}/a_n| > \vep\big) \\
&=:I_1+I_2  - I_3\,.
\end{align*}
We start with $I_2$. In the $m$-dependent case we have the following \ld\ result from \cite{mikosch:wintenberger:2016}: 
$\P(|S_{r_n}|>\sqrt{y})\sim \sigma^2\,r_n\,\P(X^2>y)$
uniformly in the  $y$-region $[r_n^{1+\delta},\infty)$, as $\nto$,
where we choose $r_n^{1+\delta}=o(a_n^2)$ and $\sigma^2$ is defined in \eqref{eq:sigmasqare}.
Then we obtain
\beao
%\dfrac{k_n}{a_n^2}\int_{r_n^{1+\delta}}^{(\vep a_n)^2}\P\big(|S_{r_n}|>\sqrt{y}\big)\,dy
I_2\sim \sigma^2\,
\dfrac{n}{a_n^2} \,\int_{r_n^{1+\delta}}^{(\vep a_n)^2} \P(X^2>y) \, dy\,,\qquad\nto\,.
\eeao 
By Karamata's integral theorem (see \cite{bingham:goldie:teugels:1987}) 
%{ \blue keep the formula for your yourself, Appendix \ref{sec:formula:tranc:moment}, %but don't quote it} 
%{\blue Karamata's theorem is also called like this in the degenerate case. It only means that the truncated second moment is \slvary . }
and 
% Therefore and by 
the definition of $(a_n)$, 
\beao
I_2&\sim & \sigma^2\,n\,\E\big[(X/a_n)^2\1(r_n^{(1+\delta)/2}<|X|\le \vep a_n) \big]\\
&\sim &\sigma^2\dfrac{ \E[X^2\1(r_n^{(1+\delta)/2}<|X|\le \vep a_n)]}{\E[X^2\1(|X|\le \vep a_n)]}\,,\qquad \nto\,.
\eeao 
The \fct\ $K(x)=\E[X^2\1(|X|\le x)]$, $x>0$, is \slvary\ and converges to infinity
as $\xto$. We choose $r_n\to\infty$ so slowly that $K(r_n^{(1+\delta)/2})/K(\vep a_n)\to 0$. Therefore and by definition of $(a_n)$, $I_2\to\sigma^2$ as $\nto$ for every $\vep>0$.

By the large derivation result we also have
$ I_3 \le c \vep^2 n\, \P(|X|>a_n \vep )\to 0$, $\nto$.
\par
Finally,
$\ell(n)=a_n/\sqrt{n}\to\infty$ and we can choose $r_n$ so small that 
\beao
I_{1}\le\dfrac{k_n}{a_n^2}\,r_n^{1+\delta}=\dfrac{r_n^\delta}{\ell^2(n)}=o(1)\,,\qquad \nto\,.
\eeao
This finishes the proof.
\end{proof}

\bpr\label{prop:cltmdep} Consider an $m$-dependent \regvary\ stationary centered 
\seq\ $(X_t)$ with index $\a=2$ and infinite variance. Then the \clt\
$a_n^{-1}S_n\std N(0,\sigma^2)$ holds as $\nto$, where $\sigma^2$ is given in 
\eqref{eq:sigmasqare}.
\epr
\begin{proof} We will apply Theorem~\ref{clt:thm:petrov}.
Lemma~\ref{lem:x2} proves \eqref{eq:6b} and provides the form of $\sigma^2$.
In the proof of Lemma~\ref{lem:x2} we used the \ld\ result in \cite{mikosch:wintenberger:2016}. It also yields \eqref{eq:6a};
cf. Remark~\ref{rem:1}. Finally, \eqref{eq:6c} follows by an application
of Lemma~\ref{lem:6c},{\bf 2.}, again by the \ld\ result in 
\cite{mikosch:wintenberger:2016} where one can choose $s_n=n^\nu$ for any $\nu>0.5$.
Moreover, $\E[|S_n|^{2-\delta}]\le c\,n^\gamma$ follows by $m$-dependence,
since $\E[X]=0$ and $\E[|S_{n}'|^{2-\delta} ]  \le c\, n\,\E[ |X|^{2-\delta}]$ where $S_n':=\sum_{t=1}^nX_t'$ for iid copies
of $X_1$; see \cite{petrov:1995}, 2.6.20 on p.~82.
\end{proof}
\bexam\label{exam:sv}{\bf An $m$-dependent stochastic volatility model.}
We consider a \regvary\ \sv\ model $X_t=\sigma_t\,Z_t$, $t\in\bbz$,
where $(\sigma_t)$ is a positive $m$-dependent 
stationary process independent of the iid 
\regvary\ centered \seq\ $(Z_t)$ with index $\a=2$. If a generic element $\sigma_0$ has a moment of order $2+\delta$ for some $\delta>0$ then $(X_t)$ is
\regvary\ stationary with index $\a=2$ and spectral tail process $\Theta_t=0$
for $t\ne 0$ and $\P(\Theta_0=\pm 1 )=\lim_{\xto}\P(\pm Z>x)/\P(|Z|>x)$; see \cite{mikosch:wintenberger:2024}, Section~5.3.1.
We can apply Lemma~\ref{lem:x2} and obtain
\beao
\sigma^2 = \E\Big[\Big(\sum_{t=0}^m \Theta_t\Big)^2\,\1(\Theta_{-m}=\cdots=\Theta_{-1}=0) \Big]=\E[\Theta_0^2]=1\,.
\eeao
Breiman's lemma (Lemma 3.1.11 in \cite{mikosch:wintenberger:2024})
yields
\beam\label{eq:an}
a_n^2 &:=&n\,K(a_n)\nonumber\\
%=n\, \E\big[(\sigma_0\,Z_0)^2\,\1(\sigma_0\,|Z_0|\le a_n)\big]\red \nonumber\\
&=&n\, \int_0^{a_n^2} \P\big(\sigma_0\,|Z_0|>\sqrt{y}\big)\,dy 
-n\,a_n^2\,\P(|\sigma_0|\,|Z_0|>a_n)\nonumber\\
&\sim &n\,\E[\sigma_0^2]\, \E\big[Z^2\,\1(|Z|\le a_n)\big]\,.
\eeam
Therefore we can choose $a_n=(\E[\sigma_0^2])^{1/2}\sqrt{n\,\E\big[Z^2\,\1(|Z|\le a_n) \big ]}=:(\E[\sigma_0^2])^{1/2}
\,a_n^Z$
where $(a_n^Z)$ is the normalizing \seq\ for the \clt\ for the iid \regvary\ 
\seq\ $(Z_t)$.
\par
The \clt\  $(a_n^Z)^{-1}S_n\std N\big(0,\E[\sigma_0^2]\big)$ will follow from Theorem~\ref{clt:thm:petrov} if we can also show that \eqref{eq:6a} and \eqref{eq:6c} are satisfied.
Condition  \eqref{eq:6a} follows from the \ld\ result for $m$-dependent \regvary\ \seq s in \cite{mikosch:wintenberger:2016}. Since we assume
$\E[Z]=0$  and independence between $(Z_t)$ and $(\sigma_t)$ we can apply 
Khintchine's inequality (see \cite{petrov:1995}, 2.6.18 on p.~82)  
conditionally on $(\sigma_t)$ and derive for $\delta>0$ 
arbitrarily small, some constant $c>0$, 
%{\blue Khintchine's inequality is applied to
%$S_n=\sigma_tZ_t$. Conditionally on $(\sigma_t)$, the $(X_t)$ are independent and centered.
%\beao
%\E\Big[\E[|S_n^p|\mid (\sigma_t)]\Big]&\le& c\,\E\Big[\E\Big[\Big(\sum \sigma_t^2Z_t^2\Big)^{p/2}\Big|(\sigma_t)\Big]\Big]\\
%&\le &\le c\,\E\Big[\Big(\sum \sigma_t^2Z_t^2\Big)^{p/2}\Big]\\
%&\le &c\, n\, \E[|\sigma^p]\E[Z|^p]
%\eeao}
\beao
\E\big[\E[|S_n|^{2-\delta}\,\big|\,(\sigma_t)\big]\big]&\le& c\,\E[|Z|^{2-\delta}]\,
\E\Big[\sum_{t=1}^n\sigma_t^{2-\delta}\Big]\le c\,n\,.
\eeao
Then we can apply Lemma~\ref{lem:6c},{\bf 2.} since $(s_n)$ can be chosen arbitrarily 
large.
\eexam
\subsection{Linear process}\label{exam:linear} 
In this section we will provide the \clt\ for a causal linear process
\beam\label{eq:infseries}
X_t=\sum_{j=0}^\infty \psi_j\,Z_{t-j}\,,\qquad t\in\bbz\,, 
\eeam
where $(Z_t)$ is an iid \regvary\ \seq\ with index $\a=2$, $\var(Z)=\infty$, 
$\E[Z]=0$, and $(\psi_j)$ is an absolutely summable  \seq\ of real numbers.
Then the infinite series in \eqref{eq:infseries} converges a.s. by virtue of the 
3-series theorem. We assume $\psi_0=1$, and it will be convenient
to write $\psi_j=0$ for $j\le -1$.
\par
The process $(X_t)$ is \regvary\ with index $\a=2$; 
see \cite{mikosch:wintenberger:2024}, Section~5.2.6.
In particular, its spectral tail process is given by
\beam\label{eq:spectal}
\Theta_t= \Theta_Z\,\dfrac{\psi_{J+t}}{|\psi_J|}\,, \qquad t\in\bbz\,,
\eeam
where $\P(\Theta_Z=\pm 1) =\lim_{\xto} \P(\pm Z>x)/\P(|Z|>x)$, $\Theta_Z$ is 
independent of $J$ with \ds\ 
$\P(J=j)=|\psi_j|^2/\|\psi\|_2^2$, $j\in\bbz$,
where $\|\psi\|_2^2:=\sum_{k=0}^\infty\psi_k^2$.
\bpr\label{prop:cltlininf} Under the aforementioned conditions, the
linear process $(X_t)$ satisfies the \clt\ $a_n^{-1} S_n\std 
N(0,\sigma^2)$ where
\beao
\sigma^2:=\Big(\sum_{j=0}^\infty \psi_j\Big)^2/\|\psi\|_2^2\,,
\eeao
and $(a_n)$ is defined in \eqref{eq:3}. Alternatively, replacing $(a_n)$ by the 
corresponding normalization $(a_n^Z)$ for the iid \seq\ $(Z_t)$, we have
$(a_n^{Z})^{-1}S_n\std N(0,\wt \sigma^2)$ where $\wt\sigma^2= \big(\sum_{j=0}^\infty \psi_j\big)^2$. 
\epr
\bre
If $\sum_{j=0}^\infty \psi_j=0$ the Gaussian limit law degenerates to zero. The result of  Proposition~\ref{prop:cltlininf} remains valid for two-sided linear 
processes  $X_t=\sum_{j\in\bbz} \psi_j\,Z_{t-j}$, $t\in\bbz$. Indeed, one can follow the lines of the proof below, first considering 
$X_t^{(m)}=\sum_{|j|\le m} \psi_j\,Z_{t-j}$, $t\in\bbz$, for fixed $m\ge 1$, and then letting $m\to\infty$. 
\ere
\begin{proof} We split the proof into two parts.
First we apply the \clt\ for an $m$-dependent 
linear process and then we extend it to the general case.\\[1mm]
{\em Finite moving averages}.
Consider the process
\beao
X_t=X_t^{(m)}:=\psi_0\,Z_t+\psi_1\,Z_{t-1}+\cdots+ \psi_m\,Z_{t-m}\,,\qquad t\in\bbz\,,\eeao 
with generic element $X^{(m)}$ and 
where $\psi_m\neq 0$. We have  assumed $\psi_0=1$, and we 
write $\psi_j=0$ for $j\in \bbz\backslash \{0,\ldots,m\}$. Then 
the random variable $J$ in 
\eqref{eq:spectal} takes values in  $\{0,\ldots,m\}$.
\par
According to \eqref{eq:spectal} and since $\psi_0=1\ne 0$, we have
\beao
\P\big(\Theta_{-m }=\cdots=\Theta_{-1}=0\big)&=&\sum_{j=0}^m
\dfrac{\psi_j^2}{\|\psi\|_2^2}\, \1\big(\psi_{j-m}=\cdots=\psi_{j-1}=0\big)\\
&=&\dfrac{\psi_0^2}{\|\psi\|_2^2}\, \1\big(\psi_{-m}=\cdots=\psi_{-1}=0\big)=
\dfrac  1 {\|\psi\|_2^2}>0\,. 
\eeao
Hence 
\beam\label{eq:x10}
\sigma^2:=\sigma_m^2&:=& \E\Big[\Big(\sum_{t=0}^m\Theta_t\Big)^2\,\1(\Theta_{-m}=\cdots=\Theta_{-1}=0)\Big]= \dfrac{\big(\sum_{t=0}^{m}\psi_t\big)^2}{\|\psi\|_2^2}
\,. 
\eeam
By Proposition~\ref{prop:cltmdep} 
we have the \clt\ $(a_n^{(m)})^{-1} S_n\std N(0,\sigma_m^2)$ 
where the normalizing \seq\  $(a_n^{(m)})$ is given by \eqref{eq:3} 
for $X=X^{(m)}$.\\[1mm]
%by \eqref{eq:spectal} if we replace $m$ by $\infty$.\\[1mm] 
Our next goal is to show that the normalizations $(a_n^{(m)})$ for 
$m\le \infty$ are \asy ally equivalent. As before, we write $X$ and $X^{(m)}$ for generic elements in the cases $m=\infty$ and $m<\infty$, respectively.
\par
We have $\P(|X^{(m)}|>x)/\P(|Z|>x)\to \sum_{j=0}^m\psi_j^2$ as $\xto$
 for $m\le \infty$; see (5.2.23) in \cite{mikosch:wintenberger:2024}. Then by Karamata's theorem and uniform \con ,
\beao\lefteqn{
\E\big[(X^{(m)}/x)^2\,\1\big(|X^{(m)}|\le x\big)\big]}\\
&=& x^{-2}\int_0^{ x^2} \dfrac{\P\big(|X^{(m)}|>\sqrt{y}\big)}{\P(|Z|>\sqrt{y})}\,
\P(|Z|>\sqrt{y})\,dy - \dfrac{\P\big(|X^{(m)}|>x \big)}{\P(|Z|>x)}\,\P(|Z|> x) \\ 
&\sim &  \|\psi\|_2^2\,\Big(
x^{-2}\int_0^{ x^2}  \P(|Z|>\sqrt{y})\, dy-\P(|Z|>x)\Big) \\
&=&\|\psi\|_2^2\;\E\big[(Z/x)^2\,\1(|Z|\le x)\big] \,,\qquad\xto\,.
%\textcolor{red}{+\|\psi\|_2^2 x^{-1}\P(|Z|>\sqrt{x})} \\
%&\sim&\|\psi\|_2^2\;\E\big[(Z/x)^2\,\1(|Z|\le x)\big]  \,.
\eeao 
Writing $a_n^{Z}$ for the solution to \eqref{eq:3} in the case $X=Z$, we have
\beao
\dfrac{n}{(a_n^Z)^2}\E\big[(X^{(m)})^2\,\1\big(|X^{(m)}|\le a_n^{Z}\big)\big] 
&\to& \|\psi\|_2^2\,, \qquad\nto\,.
\eeao 
Since $\E\big[(X^{(m)})^2\,\1\big(|X^{(m)}|\le x\big)\big]$, $x>0$, is  a \slvary\ \fct\ this implies that
\beao
\dfrac{n}{(a_n^Z \|\psi\|_2)^2}\,\E\big[(X^{(m)})^2\,\1\big(|X^{(m)}|\le 
a_n^{Z} \|\psi\|_2 \big)\big] \to 1\,,\qquad \nto\,.
\eeao
Since all normalizing constants in the \clt\ are \asy ally equivalent
we can choose $a_n^{(m)}= a_n^{Z} \|\psi\|_2$ both for $m$ finite and infinite
(we do not indicate $m$ in $\|\psi\|_2$).  
\par
In particular, by the first part of the proof we have for $m<\infty$,
\beam\label{eq:cltfin}
(a_n^{(m)})^{-1} S_n\std N\Big(0, \Big(\sum_{j=0}^m\psi_j\Big)^2/\|\psi\|_2^2\Big)\,,\qquad \nto\,.
\eeam
Letting $m\to\infty$ on the \rhs\ we derive the desired Gaussian limit \ds\ in
the case $m=\infty$.\\[1mm]
{\em The case of a causal linear process with infinitely many non-zero $\psi_j$.} 
For any finite $m$ we write $S_n^{(m)}:=\sum_{t=1}^n X_t ^{(m)}$.
We will show that for every $\delta>0$,
\beam\label{eq:billingsley}
\lim_{m\to\infty}\limsup_{\nto}\P\big(\big|S_n-S_n^{(m)}\big| >\delta\, a_n\big)=0\,.
\eeam
This together with \eqref{eq:cltfin} (and the \asy\ equivalence of the 
normalizing \seq s $(a_n)$ and $(a_n^{(m)})$)
will prove that 
\beam\label{eq:infty}
a_n^{-1}S_n\std N\Big(0,\Big( \sum_{j=0}^\infty \psi_j\Big)^2/\|\psi\|_2^2\Big)\,,\qquad \nto\,;
\eeam
see \cite{billingsley:1999}, Theorem~3.2.
By Markov's inequality we have
\beao
\P\big(\big|(S_n-S_n^{(m)})/a_n\big| >\delta \big)&\le& \delta^{-1}\,\E\big[|(S_n-S_n^{(m)})/a_n|\big]\\
&\le &c\,\sum_{j=m+1}^\infty |\psi_j|\,\E \Big[ \Big|\sum_{t=1}^n Z_{t-j}/a_n \Big|\Big ]\,.
\eeao
\begin{comment}
\textcolor{blue}{The following would be correct ?
\beao
\P\Big(\Big|(S_n-S_n^{(m)})/a_n\Big| >\delta \Big)&\le& c\,
\E\big[|(S_n-S_n^{(m)})/a_n|\big]\\
&\le &c\,\sum_{j=m+1}^\infty |\psi_j|\,
\Big(\E\Big|\sum_{t=1}^n Z_{t-j}/a_n\Big|^\gamma 
\Big)^{1/\gamma}\,.
\eeao}
\end{comment}
But the limit
$
\lim_{\nto} \E\big[\big|\sum_{t=1}^n Z_{t}/a_n\big|\big]
$
exists and is finite. Indeed, the \clt\ for the iid \seq\ $(Z_t)$ with 
normalization $a_n$ and normal limit together with 
dominated \con\ ensure that this limit exists.
We conclude that \eqref{eq:billingsley} holds. We have proved \eqref{eq:infty}
and the first statement of Proposition~\ref{prop:cltlininf}. The second statement
about the change of normalization $(a_n)$ by $(a_n^Z)$ is straightforward by the aforementioned discussion.
\end{proof}

\subsection{A \sv\ model}\label{sec:svmodel}
Recall the definition of a \sv\ process from Example~\ref{exam:sv}: 
$X_t=\sigma_t\,Z_t$, $t\in\bbz$, where $(\sigma_t)$ is a positive 
stationary volatility \seq\ with $\E[\sigma_0^{2+\delta}]<\infty$ for some $\delta>0$,
independent of the iid \regvary\ centered noise $(Z_t)$ with index $\a=2$ and 
infinite variance. 
%In addition, we also assume  that a generic element $Z$ is symmetric.
We do not require that the volatilities are $m$-dependent but we assume 
the concrete structure 
\beam\label{eq:formsig}
\sigma_t=\exp\Big(\sum_{j=0}^\infty \psi_j\,\eta_{t-j}\Big)\,,\qquad t\in\bbz\,,
\eeam 
where $(\eta_j)$ is iid standard normally distributed and the real coefficients
$(\psi_j)$ are square-integrable. This model has been studied intensively; see \cite{breidt:davis:1998}.
The so defined \sv\ process $(X_t)$ is
\regvary\ with index $\a=2$ and  has spectral tail process $\Theta_t=0$ for $t\ne 0$, 
$\P(\Theta_0=\pm 1)=\lim_{\xto}\P(\pm Z>x)/\P(|Z|>x)$; see Section~5.3.1 in \\ \cite{mikosch:wintenberger:2024}.
In view of Example~\ref{exam:sv} we have the \clt\ $(a_n^{(m)})^{-1}S_n^{(m)}\std 
N(0,1)$ where $X_t^{(m)}=\sigma_t^{(m)}Z_t$, $\sigma_t^{(m)}$ is given by
 \eqref{eq:formsig} with the convention that $\psi_j=0$ for $j>m$ and 
$a_n^{(m)}= \big(\E[(\sigma_0^{(m)})^2]\big)^{1/2}\,a_n^{Z}$, and $(a_n^Z)$ 
is defined in Example~\ref{exam:sv}. Alternatively, we have
\beao
(a_n^{Z})^{-1} 
S_n^{(m)}\to N\big(0, \E[(\sigma_0^{(m)})^2]\big)= N\Big(0,\exp\big(2\sum_{j=0}^m\psi_j^2\big)\Big)\,,\qquad \nto\,.
\eeao
Letting $m\to\infty$ on the \rhs , we obtain the Gaussian limit for $(S_n/a_n^{Z})$, provided we can show that for every $\gamma>0$,
\beao
\lim_{m\to\infty}\limsup_{\nto}\P\big(|S_n-S_n^{(m)}|>\gamma\,a_n^Z\big)=0\,.
\eeao
This is our next goal.
Write $\ov Z_t=Z_t\,\1(|Z_t|\le a_n^Z)$.
\beao\lefteqn{
\P\big(|S_n-S_n^{(m)}|>\gamma\,a_n^Z\big)}\\&\le& n\,\P(|Z|>a_n^Z)+
\P\Big(\Big|\sum_{t=1}^n (\ov Z_t- \E[\ov Z])\,(\sigma_t-\sigma_t^{(m)})
\Big|>\gamma\,a_n^Z/2\Big)\\&&
+\P\Big(n\,\E[|Z/a_n^Z|\1(|Z|> a_n^Z)]\,\Big|\dfrac{1}{n}\sum_{t=1}^n(\sigma_t-\sigma_t^{(m)})\Big| >\gamma/2\Big)\\
&=:&I_1+I_2+I_3\,.
\eeao  
We have $I_1\to0$ as $\nto$ by definition of $(a_n^Z)$. 
We observe for fixed $m$ that by the ergodic theorem,
\beao
\dfrac{1}{n}\sum_{t=1}^n(\sigma_t-\sigma_t^{(m)})&\stas& 
\E[\sigma_0-\sigma_0^{(m)}]
%=\exp\Big({\red 0.5}\,\sum_{j=0}^m\psi_j^2\Big)\,\Big(\exp\Big({\red 0.5}\,\sum_{j=m+1}^\infty\psi_j^2\Big)-1\Big)
\,,\qquad \nto\,.
\eeao
%and the \rhs\ vanishes as $m\to\infty$. 
Moreover, by Karamata's theorem and the definition of $(a_n^Z)$, 
$n\,\E[|Z/a_n^Z|\1(|Z|> a_n^Z)]\to 0$ as $\nto$. Hence for every $m\ge1$,
$I_3\to 0$ as $\nto$.
To show that $I_2$ is negligible, we condition on $(\eta_t)$ and apply \v Cebyshev's inequality:
\beao
I_2&\le&% \P\Big(\Big|\sum_{t=1}^n (\ov Z_t-\E[Z\,\1(|Z|\le a_n^Z)])\,(\sigma_t-\sigma_t^{(m)})
%\Big|>\gamma\,a_n^Z/2\Big)\\
%&\le &
c\,n\,\E[ (Z/a_n^Z)^2\,\1(|Z|\le a_n^Z)]\,\E\big[(\sigma_0-\sigma_0^{(m)})^2\big]\,.
\eeao
The first factor is bounded by definition of $(a_n^Z)$ while 
$\E\big[(\sigma_0-\sigma_0^{(m)})^2\big]\to 0$ as $\mto$. 
Thus we have proved the following result. 
\bpr
Consider the \sv\ model $X_t=\sigma_t\,Z_t$, $t\in\bbz$, where $(Z_t)$ is centered 
iid \regvary\ with index $\a=2$ and infinite variance independent
of the log-volatility \seq\ $\log\sigma_t=\sum_{j=0}^\infty \psi_j\,\eta_{t-j}$,
$t\in\bbz$, where $(\psi_j)$ is a square-integrable  \seq\ of real numbers and 
$(\eta_t)$ is iid standard normal. Then the following \clt\ holds
\beam\label{eq:mgfvar}
(a_n^Z)^{-1} S_n\std N\Big(0,\exp\Big(2\,\sum_{j=0}^\infty \psi_j^2\Big)\Big)\,,\qquad \nto\,.
\eeam
\epr  
\bre
The condition of iid standard normal noise $(\eta_t)$ can be relaxed. It suffices that
$\eta_0$ has a finite \mgf\ \st\ $\E[\sigma_0^2]= \prod_{j=0}^\infty
\E[\ex^{2\,\psi_j\,\eta_0}]$ is finite. Then the \asy\ variance in 
\eqref{eq:mgfvar} depends on the concrete form of this \mgf .
\ere
\bre
\par
Consider an iid standard normal \seq\ $(N_t)$ independent of $(X_t)$. 
We observe that the process $\wt X_t=X_t\,N_t=\sigma_t\,(Z_t\,N_t)$, $t\in\bbz$,
is again a \regvary\ stationary \sv\ process with index $\a=2$. Indeed,
$(Z_t\,N_t)$ is iid symmetric and by Breiman's lemma 
(Lemma 3.1.11 in \cite{mikosch:wintenberger:2024}),
\beao
\P(|Z_0\,N_0|>x)\sim \E[N_0^2] \,\P(|Z|>x)=\P(|Z|>x)\,,\qquad \xto\,.
\eeao
Choosing the normalizing constants according to
\eqref{eq:an}, we have $a_n^{-1}\sum_{t=1}^nX_t\,N_t\std N(0,1)$.
According to Lemma~\ref{lem:A1} this relation holds \fif\ 
$a_n^{-2}\sum_{t=1}^nX_t^2=:a_n^{-2} V_n^2 \stp 1$. Thus we also have 
the \clt\ for the studentized sum $S_n/V_n\std N(0,1)$ as $\nto$.
\ere

\subsection{Affine \sre s: the Kesten-Goldie case}\label{sec:sre1}
In contrast to the previous examples (linear process, \sv\ model),
in this section  we will consider 
a stationary \spr\ $(X_t)$ whose \regvar\ property is not caused 
by the iid \regvary\ innovations: the stationary solution to a particular affine \sre .
The marginal tails of this process are almost of Pareto-type, hence the marginal
\ds\ has infinite variance if $\a=2$.
We will show that the \clt\ for this process can be 
explained by its \ld\ properties.
\par
We consider the stationary solution to the \sre
\beam\label{eq:sre1}
X_t=A_t\,X_{t-1}+B_t\,,\qquad t\in\bbz\,,
\eeam  
under the following conditions on the $\bbr_+\times\bbr$-valued iid \seq\ $(A_t,B_t)$, $t\in\bbz$, with generic element $(A,B)$:
\begin{itemize}
\item
The equation $\E[A^\a]=1$ has the solution $\a=2$, $m_2:=\E[A^2\,\log A]<\infty$ and $\E[B^2]<\infty$.
\item The law of $\log A$ conditioned on $\{A>0\}$ is non-arithmetic and for every $x$,  $\P(A\,x+B=x)<1$.
\end{itemize} 
Under these conditions, \eqref{eq:sre1} has a unique solution $(X_t)$ with generic element $X$, the process is \regvary\
with index $\a=2$, and the tails of the marginal \ds\ satisfy the \asy\ relation
\beao
\P(\pm X>x)\sim c_{\pm}\,x^{-2}\,,\qquad \xto\,,
\eeao
where $m_2>0$ and the constants $c_{\pm}$ are given by
\beao
 c_{\pm}:=\E\big[(A\,X+B)_{\pm}^2-(A\,X)_{\pm}^2\big]/(2\,m_2)\,.
\eeao
In particular, since $\E[X]=\E[X]\,\E[A]+\E[B]$, as $\xto$,
\beao
x^2\,\P(|X|>x)
\to c_++c_-&=&
\dfrac{\E\big[(A\,X+B)^2-(A\,X)^2\big]}{2\,m_2}=
\dfrac{\E[B^2+2\,A\,B\,X]}{2\,m_2}\\&=&\dfrac{\E[B^2]\,(1-\E[A])+2\,\E[A\,B]\,\E[B]}{2\,(1-\E[A])\,m_2}=:c_\infty\,.
\eeao
When we choose $B=1$ a.s. the constant $c_\infty$ turns into
\beam\label{eq:c0}
c_0:= \dfrac{1+\E[A]}{1-\E[A]}\,\dfrac{1}{2\,m_2}\,.
\eeam
\par
These results are known as the Kesten-Goldie theorem; see \cite{kesten:1973}, \cite{goldie:1991}; cf. Theorem 2.4.4 in \cite{buraczewski:damek:mikosch:2016}. 
\par
In the following result we assume slightly stronger assumptions on the moments
of $A$ and $B$ as well as suitable mixing conditions. The additional moment
conditions are needed in the \ld\ results of Lemma~\ref{lem:ldsre} below.
\bpr\label{prop:sreclt} Assume the aforementioned conditions on the stationary solution $(X_t)$ to
the affine \sre\ \eqref{eq:sre1}. In addition, we require 
$\E[A^\gamma+|B|^\gamma]<\infty$ for some 
$\gamma>2$ and that $(X_t)$ is strongly mixing with geometric rate. 
Then 
the \clt\ 
$a_n^{-1}(S_n-d_n)\std N(0,c_0)$ holds where $d_n=n\,\E[X]=n\,\E[B]/(1-\E[A])$, 
$(a_n)$ satisfies \eqref{eq:3} and $c_0$ is defined in \eqref{eq:c0}.
\epr
\bre
Since $\P(|X|>x)\sim c_{\infty}\,x^{-2}$ we can choose 
the normalization $a_n=\sqrt{c_\infty\,n\,\log n}$; see the proof of the proposition below.
Then the \clt\ can be reformulated as follows:
\beao
\dfrac{S_n-n\, \E[B] /(1-\E[A])}{ \sqrt{n\,\log n}}\std N(0,c_0\,c_\infty)\,. 
\eeao
\ere
\bre Under mild conditions
the solution $(X_t)$ to an affine \sre\ is $\beta$-mixing with geometric rate,
hence strongly mixing with geometric rate. For example,
the geometric rate is achieved if $A_t=A(Z_t)$, $B_t=B(Z_t)$, $t\in\bbz$,
are polynomials of an iid \seq\ $(Z_t)$, and a generic element of $Z$ 
has a Lebesgue density on some open interval containing zero; 
see \cite{mokkadem:1990}.
For further discussions on this topic we refer to Section~4.2.2 in 
 \cite{buraczewski:damek:mikosch:2016}. We will apply strong mixing  with geometric rate to achieve the mixing condition  \eqref{eq:5} for suitable $(r_n)$.
An alternative way
of verifying \eqref{eq:5} is by using the fact that 
$(X_t)$ is a \MC\ satisfying a contractive \sre . In this case, coupling arguments
can be applied; see Section~9.1.3 
in \cite{mikosch:wintenberger:2024}, in particular the proof
of Theorem~9.3.6,
 where the $\a$-stable \clt\ is derived in the case $\a<2$.
\ere
\bre \cite{jakubowski:swewczak:2020} proved a \clt\ for a stationary \garch\ process $(X_t)$ in the infinite variance case. 
These processes are known to have tails of the form $\P(\pm\,X>x)\sim c_{\pm}\,x^{-\a}$, $\xto$, under mild conditions on the underlying innovations \seq ; see \cite{mikosch:starica:2000}; cf. \cite{buraczewski:damek:mikosch:2016}. The squared volatility process of  $(X_t)$ satisfies an affine 
\sre\ of Kesten-Goldie type, and the tails of $(X_t)$ inherit the power-law tails from the latter process. The results in \\ 
\cite{jakubowski:swewczak:2020} focus on
the case $\a=2$. They use elegant martingale techniques providing the \clt\ for $(S_n/a_n)$ and the law of large numbers\\ $V_n^2=\sum_{t=1}^nX_t^2/a_n^2\stp 1$.
Thus $S_n/V_n\std N(0,1)$ as $\nto$. This approach avoids knowledge of the normalizing constants $a_n$.

\ere
\begin{proof}
We start by observing that the mixing condition \eqref{eq:5} holds, and we can choose
%$\red r_n=\log_{k-1} n$ for any $k\ge 2$. 
$\ell_n= [C \log n]$ for a suitable $C>0$ and $r_n=[(\log n)^{1+\vep}]$ for arbitrarily small $\vep>0$. 
This 
follows from the strong mixing condition with geometric rate for $(X_t)$; 
see Example~\ref{exam:mixing} for a discussion on the mixing properties of $(X_t)$.
\par
We will apply Theorem~\ref{clt:thm:petrov} and the following \ld\ result in 
the particular case $\a=2$ taken from 
\cite{buraczewski:damek:mikosch:zienkiewicz:2013}, Theorem~2.1. 
\ble\label{lem:ldsre} Under the assumptions of Proposition~\ref{prop:sreclt} we have for $M>2$, 
\beam\label{eq:help0}
 \sup_n \sup_{n^{1/2}(\log n)^M \le x } \dfrac{\P(|S_n-d_n|>x)}{n\,\P(|X|>x)}<\infty\,,
\eeam
and, if $(t_n)$ is chosen \st\
$t_n\ge \log\big(n^{1/2} (\log n)^M\big)$ for large $n$ 
and $\lim_{n\to\infty} t_n/n=0$ then 
\beam\label{eq:help1}
\lim_{n\to\infty} \sup_{n^{1/2}(\log n)^M\le x \le \exp(t_n)} \Big|
\dfrac{\P(|S_n-d_n|>x)}{n\,\P(|X|>x)}-c_0\Big|=0\,.
\eeam
The constant $c_0$ is defined in \eqref{eq:c0}.
\ele
We start by verifying \eqref{eq:6a} in Theorem~\ref{clt:thm:petrov}.
We choose $(r_n)$ \st\ $a_n^2=c_\infty\,n\,\log n>r_n\,(\log r_n)^{2\,M}$ for some $M>2$.
This is always possible since $r_n/n=o(1)$. By \eqref{eq:help0} we have
\beam\label{eq:3x}
k_n\,\P\big(|S_{r_n}-d_{r_n}|>a_n\big)\,\le c\,n\,\P(|X|>a_n)\to 0\,,\qquad \nto\,.
\eeam
\par
Next we verify \eqref{eq:6c}. Since $\E[S_{r_n}-d_{r_n}]=0$ it suffices to
show that the following quantity vanishes as $\nto$: 
\beao
v_n&:=&\dfrac{k_n}{a_n}\,\E\big[\big|S_{r_n}-d_{r_n}\big|\1\big(|S_{r_n}-d_{r_n}|>a_n\big)\big]\\
&=&\dfrac{k_n}{a_n} \int_{a_n}^\infty \P\big(\big|S_{r_n}-d_{r_n}\big|>y\big)\,dy  + k_n\,\P\big(|
S_{r_n}-d_{r_n}|>a_n \big)\,.
\eeao
\begin{comment}
\textcolor{blue}{
I could not find wrong points in the following calculation 
\begin{align*}
 \E[|X|\1 (|X|>a_n)] &= \int_0^\infty \1(x>a_n) x\P(|X|\in dx) \\
&= \int_{0}^\infty \1(x>a_n) \int_0^\infty \1(y <x) dy \P(|X|\in dx) \\
&= \int_0^\infty \int_0^\infty \1(y \vee a_n <x) \P(|X|\in dx)dy \\
&= \int_0^\infty dy \P(y \vee a_n < |X|) \\
&= \int_0^{a_n} \P(a_n<|X|) dy  + \int_{a_n}^\infty dy \P(y<|X|) \\
& = \int_{a_n}^\infty \P(|X|>y) dy + a_n\P(|X|>a_n).
\end{align*}
}
\end{comment}
Due to the large deviation result the second term is bounded by 
$c\,n\,\P(|X|>a_n)\to 0$ as $n\to \infty$. 
By the choice of $(r_n)$ and by virtue of \eqref{eq:help0} we have
\beao
v_n&\le& c\,\dfrac{k_n}{a_n}\,r_n\,\int_{a_n}^\infty \P(|X|>y)\,dy\\
&\le&c\,\dfrac{n}{a_n}\,\int_{a_n}^{\infty} y^{-2}\,dy=O(n/a_n^2)=o(1)\,,\qquad \nto\,.
\eeao
\par
Finally, we prove \eqref{eq:6b}. We have 
\beao
\lefteqn{\dfrac{k_n}{a_n^2} \E\big[(S_{r_n}-d_{r_n})^2\,\1\big(|S_{r_n}-d_{r_n}|\le \vep\,a_n\big)\big]}\\
&=&\dfrac{k_n}{a_n^2}\Big(\int_0^{r_n\,(\log r_n)^{2\,M}} + \int_{r_n\,(\log r_n)^{2\,M}}^{(\vep\,a_n)^2}\Big) \P\big(|S_{r_n}-d_{r_n}|>\sqrt{y}\big)\,dy \\
&& - \vep^2 k_n\P\big(|S_{r_n}-d_{r_n}|>\vep a_n \big) 
=:I_1+I_2-I_3\,.
\eeao
From \eqref{eq:3x} we also have $I_3\to 0$ as $n\to\infty$. 
By assumption, $M>2$. Then we can use the \ld\ result \eqref{eq:help1} to obtain
\beao
I_2&\sim& \dfrac{n}{a_n^2}\,c_0\,
\int_{r_n\,(\log r_n)^{2\,M}}^{(\vep a_n)^2}\P(|X|>\sqrt{y})\,dy\\
&\sim & \dfrac{n}{a_n^2}\,c_0\,c_\infty \log (a_n^2)\to c_0\,,\qquad \nto\,.  
\eeao
We also have by definition of $(r_n)$,
\beao
I_1&\le &  \dfrac{k_n}{a_n^2}\,r_n\,(\log r_n)^{2\,M}\le c\,\dfrac{(\log r_n )^{2\,M}}{\log n}\to 0\,,\qquad \nto\,.
%\dfrac{(\log _{k-1})^{2\,M}}{\log n}\to 0\,,\qquad \nto\,.
\eeao
\end{proof}
\subsection{Affine \sre s: the Grincevi\v cius-Grey case}\label{subsec:sre2}
The Kesten-Goldie theory is supplemented by the Grincevi\v cius-Grey theory.
Here one considers the stationary 
solution to the affine \sre\ \eqref{eq:sre1} under the conditions
that $(A_t,B_t)$, $t\in\bbz$, is an $\bbr_+^2$-valued iid \seq\ \st\ 
a generic element $(A,B)$ satisfies $\E[A^\a]<1$ and $B$ is \regvary\ with index
$\a>0$. Then $(X_t)$ is \regvary\ with index $\a$; see \cite{mikosch:wintenberger:2024}, Proposition 5.6.10. The stationary solution has \rep\ 
$\sum_{j=-\infty}^t\,B_j\,A_{j+1}\cdots\,A_t$, $t\in\bbz$. The \regvar\
of $(X_t)$ is inherited from the \regvar\ of the iid \seq\ $(B_t)$. This is in stark
contrast to the Kesten-Goldie situation where the \regvar\ of a single 
element $X_t$ (and of the whole \seq ) is due to its infinite series \rep .
In other words, in the latter case the truncated $m$-dependent series 
$X_t^{(m)}=\sum_{j=t-m}^t\,B_j\,A_{j+1}\cdots\,A_t$ is not \regvary\ with index $\a$
while, in the Grincevi\v cius-Grey case, $(X_t^{(m)})$ is \regvary\ due to the 
\regvar\ of $B$. The Grincevi\v cius-Grey case has some similarity with
linear processes where the \regvar\ is inherited from the \regvar\ of the iid
noise. 
\par
In addition to the aforementioned conditions we assume:
\begin{itemize}
 \item $(A_t)$ and  $(B_t)$ are independent.
\item $B$ is regularly varying with index $\alpha=2$, $\E[ A^2]<1$ and $\E[ A^{4}]<\infty$. %{\blue Do we need this ? Yes, it is a condition in Thm 4.2
%in M. and Konstantinides.} 
%it is related with conditions $\beta\in (2,4]$ \st\  $\E[ A^\beta]=1$ in Corollary %\ref{cor:large:deviation:grey}, but looking the proof it seems that 
%we could relax this condition ?}. 
\end{itemize}
Consider the unique solution to the affine \sre\ $C_t=1+A_t\,C_{t-1}$, $t\in\bbz$, given by
\beam\label{def:c}
 C_t =\sum_{j=-\infty}^t A_{j+1}\cdots A_t\,,\qquad t\in \bbz\,,
\eeam
write $C$ for a generic element and $\sigma_x^2:=\var (B\,\1(B\le x))$, $x>0$.
\par
The following large derivation result  is taken from \cite{konstantinides:mikosch:2005}, Theorem~4.2.
\bpr\label{thm:large:derivation:GGcase}
Assume the aforementioned conditions in the Grincevi\v cius-Grey case.
Consider a sequence of positive numbers $(x_n)$ 
such that $n\,\P(B>x_n) \to 0$ and 
\beam\label{eq:large}
&&\lim_{n\to\infty}\sup_{x\ge x_n} \dfrac{\P\Big(\Big(
\sum_{t=1}^n C_t^2\Big)^{1/2} >x/(\sqrt{\log x}\,\sigma_x)
\Big)+\P\Big(\Big|\sum_{t=1}^n (C_t-\E C)\Big|>x \Big)}{n\,\P(B>x)}=0\,. \nonumber\\
\eeam
Then we have 
\beao
 \lim_{n\to\infty}\sup_{x\ge x_n} \Big|
\dfrac{\P(S_n-\E S_n >x)}{n\,\P(B>x)}- \E[ C^{2}] 
\Big|= \lim_{n\to \infty} \sup_{x \ge x_n} \dfrac{\P(S_n-\E S_n \le -x)}{n\,\P(B>x)}=0 \,.
\eeao 
\epr
\bre We observe that $\E[C]=1+\E[C]\,\E[A]$ and 
$\E[C^2]=1+\E[A^2]\,\E[C^2]+2\,\E[A]\,\E[C]$. Hence
\beao
\E[C]=\dfrac{1}{1-\E[A]}\quad \mbox{ and }\quad \E[C^2]=\dfrac 1 {1-\E[A^2]}\,\dfrac{1+\E[A]}{1-\E[A]}\,.
\eeao
\ere
\begin{corollary}
\label{cor:large:deviation:grey}
Assume the conditions of  Proposition~\ref{thm:large:derivation:GGcase} on the 
solution $(X_t)$ to the \sre\ \eqref{eq:sre1}
and
the Kesten-Goldie conditions for the unique solution to the 
\sre\ $C_t=A_t\,C_{t-1}+1$, $t\in\bbz$: there exists $\beta>2$ \st\  $\E[ A^\beta]=1$, 
$\E[A^\beta\,\log A]<\infty$, and the law of $\log A$ conditioned on $\{A>0\}$ is non-arithmetic.
Then the following \ld\ relation holds for every $\delta>0$:
\beam\label{const:cor:large:deviation:gray}
 \lim_{n\to\infty} \sup_{x\ge n^{0.5+\delta}} \Big|
\dfrac{\P(|S_n-\E S_n|>x)}{n\,\P(X>x)}- {\dfrac{1+\E[A]}{1-\E[A]}}
\Big| =0 \,,\qquad \nto\,.
\eeam
\end{corollary}
\bre The Kesten-Goldie conditions ensure that $(C_t)$ is strictly stationary and \regvary\ with index $\beta>0$. In particular, 
$\P(C>x)\sim c\,x^{-\beta}$ as $\xto$ where
\beao
 c=\dfrac{\E\big[(1+C\,A)^{\beta}-(C\,A)^\beta\big]}{\beta\, \E[A^\beta\,\log A]}\,.
\eeao
\ere
\begin{proof}
By (5.6.11) in \cite{mikosch:wintenberger:2024},  $X$ and $B$ are tail-equivalent, i.e.,
\beam\label{eq:tail:equiv}
\P(X>x)\sim \P(B>x)/{ (1-\E [A^2])}\,,\qquad  \xto\,.
\eeam 
Therefore it suffices to check the conditions of Proposition~\ref{thm:large:derivation:GGcase}. We choose
$x_n:= n^{0.5+\delta}$. Since $B$ is \regvary\ with index $\a=2$ we have
$n\,\P(B>n^{ 0.5+\delta})\to 0$ as $\nto$.
\par
Lemma~\ref{eq:help0} (Theorem 2.1 in \cite{buraczewski:damek:mikosch:zienkiewicz:2013}) yields a \ld\ result
in the Kesten-Goldie case for the solution to the \sre\ $C_t=A_t\,C_{t-1}+1$: 
for every $\delta>0$,
\beao
 \limsup_{\nto}\sup_{x\ge n^{0.5+\delta}}  
\dfrac{\P\Big(\Big|\sum_{t=1}^n (C_t-\E[ C])\Big|>x \Big)}{n\,\P( C>x)}<\infty. 
\eeao
Since $B$ and $C$ are \regvary\ with indices $2$ and $\b$, respectively, and $2<\b$, 
we have $\lim_{\xto}\P(C>x)/\P(B>x)=0$, hence
\beam\label{eq:largo}
\lim_{\nto}\sup_{x\ge n^{0.5+\delta}}  
\dfrac{\P\Big(\Big|\sum_{t=1}^n (C_t-\E[ C])\Big|>x \Big)}{n\,\P( B>x)}=0\,.
\eeam
This proves that the second part in \eqref{eq:large} vanishes.
\par
%Now we enlarge the threshold to $x_n=n^{1+\delta}$. 
 Next we bound the first part
in \eqref{eq:large}. Choose $\gamma\in (2,\beta)$. By Jensen's inequality,
\beao
\E \Big[\Big(\sum_{t=1}^n C_t^2\Big)^{\gamma/2} \Big]
& = &
n^{\gamma/2}\, \E \Big[\Big(\sum_{t=1}^n C_t^2/n\Big)^{\gamma/2}\Big]\le n^{\gamma/2}\, \E \big[C^{\gamma}\big]\,.
\eeao
Then, also observing that  $\P(B>x)=x^{-2}L_0(x)$ for some 
\slvary\ $L_0$, an application of Markov's inequality yields for $x\ge x_n=n^{0.5+\delta}$,
%{\red the power $2\delta$ is incorrect. The \slvary\ \fct\ is bounded by $n^a$ for a small
%positive value $a$. One loses one $\delta$.}
\beao
\lefteqn{\dfrac{\P\Big(\Big(\sum_{t=1}^n C_t^2\Big)^{0.5} > x/(\sqrt{\log x}\,\sigma_x)\Big)}{n\,\P(B>x)}}\\
& \le&c\,(n/x^2)^{\gamma/2-1}\,
\dfrac{ (\sqrt{\log x}\,\sigma_x)^\gamma}{L_0(x)}
\le c\,n^{ -\delta\,(\gamma/2-1)}\,,
\eeao
where in the last step we also applied the Potter bounds (see \cite{bingham:goldie:teugels:1987})
to the \slvary\ \fct\ $ (\sqrt{\log x}\,\sigma_x)^\gamma/L_0(x)$.

\begin{comment}
Therefore, noticing that $\P(B>x)\sim x^{-2}\ell(x)$, we have  
\begin{align*}
 \frac{\P\Big(
\Big(\sum_{t=1}^n C_t^2 \Big)^{\gamma /2}> x^\gamma /(\sqrt{\log x}\,\sigma_x)^\gamma \Big)} 
{n\P(B>x)} 
& \le \E [C^{\gamma}]
\frac{n^{\gamma/2-1} (\sqrt{\log x}\,\sigma_x)^\gamma}{x^{\gamma-2}\ell(x)}
\end{align*}
%\beao\lefteqn{
%\P\Big(
%\Big(\sum_{t=1}^n C_t^2 \Big)^{\gamma /2}> x^\gamma /(\sqrt{\log x}\,\sigma_x)^\gamma \Big)}\\ 
%& \le &\P\Big(
%\sum_{t=1}^n C_t > x/(\sqrt{\log x}\,\sigma_x)\Big)\\
%& \le &\P\Big(\sum_{t=1}^n (C_t -\E [C])> 0.5\,x /(\sqrt{\log x}\,\sigma_x)
%\Big) + \1\big( n\, \E [C ]> 0.5\,x /(\sqrt{\log x} \,\sigma_x)
%\big)\,.
%\eeao
%But $\sqrt{\log x}\,\sigma_x$ is \slvary , hence 
%the second term vanishes for $x>x_n$, and
%we can also apply \eqref{eq:largo} to the first term: when divided by $n\,\P(B>x)$
%it vanishes 
%uniformly for $x\ge n^{1+\delta}$.  This proves \eqref{eq:large}.
But $\ell(x) =o(\sigma_x)^\gamma$ and $(\sqrt{\log x}\,\sigma_x)^\gamma /\ell(x)$ is slowly varying, and moreover 
$x_n^{\gamma-2} \ge n^{(\gamma-2)(0.5+\delta)}=n^{\gamma/2-1+(\gamma-2)\delta}$. 
\end{comment}

Since the \rhs\ converges to zero as $\nto$, \eqref{eq:large} holds.
The conditions of Proposition~\ref{thm:large:derivation:GGcase} are satisfied and  the corollary follows. 
\end{proof}

Now we turn to the \clt\ in the borderline case when 
$X$ is \regvary\ with index $\a=2$ and $\var(X)=\infty$. Since $X$ and $B$ are tail-equivalent (see \eqref{eq:tail:equiv}) this means that $B$ is \regvary\ with index 
$\a=2$ and $\var(B)=\infty$.

\begin{proposition}
\label{prop:main:sre2}
 Assume the conditions of Corollary \ref{cor:large:deviation:grey} on the stationary solution $(X_t)$ to the stochastic recurrence equation \eqref{eq:sre1}
in the case $\a=2$, $\var(B)=\infty$, that $(X_t)$ is strongly mixing 
with geometric rate and \st\ we can choose the block length $r_n$ so small
that $(n/a_n^2)\,r_n^{\delta}=o(1)$ as $n \to \infty$ for some $\delta\in (0,1)$.
Then the central limit theorem $a_n^{-1}(S_n-d_n) \std  N(0,c_0)$ holds where
$(a_n)$ is defined in \eqref{eq:3}, $d_n=n \E[X]=n\,\E[B]/(1-\E[ A])$ and $c_0=(1+\E[A])/(1-\E[A])$.
\end{proposition}
\bre\label{rem:4x}
Since $\var(B)=\infty$ the \seq\ $(a_n)$ satisfies $a_n=\sqrt{n}\ell(n)$ for a \slvary\ \fct\ $\ell(n)\to\infty$. Therefore the growth condition on $(r_n)$  means that it must be chosen \st\ $r_n^{\delta}/\ell^2(n)\to0$ for an arbitrarily small $\delta>0$.
The required relation
$(n/a_n^2) r_n^\delta=r_n^\delta/\ell^2(n)=o(1)$ as $\nto$ 
between the block lengths $r_n$ and the normalizing constants $a_n$ is delicate. 
For example, for a Pareto-like tail $\P(B>x)\sim c_1\,x^{-2}$  as $\xto$ for some $c_1>0$ 
we have $K(x)=\E[X^2\1(|X|\le x)] \sim c_2 \,\log x$ as $\xto$ for some constant $c_2>0$, hence 
$\ell^2(n)\sim  (c _2\,\log n)/2$. Then, choosing $r_n\sim (\log n)^{(1-\vep) /\delta}$ for small  $\vep>0$, the condition $r_n^\delta/\ell^2(n)=o(1)$ in Proposition~\ref{prop:main:sre2} is satisfied, while the mixing condition \eqref{eq:3a} is kept.
\ere
\begin{comment}
{\blue But then 
$\var(B)<\infty$.} 
??? What is the purpose of this remark?} 
However, if $\P(B>x)\sim c(x\log x)^{-2}$ and then $K(x)\sim c/2 \log\log x$. Thus for any $\delta>0$, $r_n^\delta/\ell^2(x) \to \infty$ and 
the condition of Proposition \ref{prop:main:sre2} is violated. Moreover, notice that 
{\blue $X$ and $B$ are tail-equivalent ?}
$\P(X>x)\sim c(x (\log x)^{1+\vep})^{-1}$ for arbitrarily small $\vep>0$ implies $\E[X^2]<\infty$.
\end{comment} 

\begin{proof}
Fix $\delta>0$ \st\ $r_n^{\delta}/\ell^2(n)\to0$. Then the condition
$ \vep a_n > r_n^{(1+\delta)/2}$ is trivially satisfied and we can apply
\eqref{const:cor:large:deviation:gray}: 
\beam\label{eq:4x}
 k_n\,\P \big(|S_{r_n}-d_{r_n}|>a_n\big) \le c\, n\,\P(X>a_n) \to 0,\qquad \nto\,.
\eeam
The \con\ to zero follows from the definition of $(a_n)$ in the case $\a=2$; see the comment after \eqref{eq:3a}. Thus \eqref{eq:6a} holds.
\par
Next we prove \eqref{eq:6c}. 
Since $\E[S_{r_n}-d_{r_n}]=0$ it suffices to show that the following quantity vanishes as $n\to \infty$. 
\beao
 v_n&:=& \dfrac{k_n}{a_n} \E \big[|S_{r_n}-d_{r_n}|\,\1\big(|S_{r_n}-d_{r_n}|> a_n\big)\big] \\
& = &\dfrac{k_n}{a_n} \int_{a_n}^\infty \P\big(
|S_{r_n}-d_{r_n}|>y
\big)\, dy + k_n\, \P\big(|S_{r_n}-d_{r_n}|> a_n\big)\,,
\eeao
where the second term vanishes as $\nto$ by virtue of \eqref{eq:4x}. From
\eqref{const:cor:large:deviation:gray} and  Karamata's theorem we also obtain
\beao
 v_n  \le c\, \dfrac{k_n}{a_n}\, r_n \int_{a_n}^\infty \P( X >y)\,dy  +o(1) \le c\,n\, \P( X >a_n) +o(1)\to 0\,,\qquad \nto\,.
\eeao
\par
Finally, we prove \eqref{eq:6b}. We have for every $\vep>0$, by virtue of \eqref{eq:4x},
\beao
\lefteqn{\dfrac{k_n}{a_n^2}\, \E \big[
 (S_{r_n}-d_{r_n})^2 \,\1\big(|S_{r_n}-d_{r_n}| \le \vep\, a_n\big)
\big] }\\
& =& \dfrac{k_n}{a_n^2}\, \int_0^{(\vep a_n)^2} \P\big(
|S_{r_n}-d_{r_n}|>\sqrt{y}
\big)\, dy {-}   k_n\, \vep^2\, \P\big( |S_{r_n}-d_{r_n}| >\vep\,a_n\big) \\
& =& \dfrac{k_n}{a_n^2} \Big(
\int_{0}^{r_n^{1+\delta}} + \int_{r_n^{1+\delta}}^{(\vep a_n)^2} \Big) \P\big(
|S_{r_n}-d_{r_n}|>\sqrt{y}
\big) \,dy +o(1)\\
& =:& I_1+I_2+o(1)\,.
\eeao
By assumption 
we have chosen $\delta>0$ \st 
\beao
I_1 \le \dfrac{k_n}{a_n^2} r_n^{1+\delta} \le 
\dfrac{r_n^{\delta}}{\ell^2(n)}\to  0\,,\qquad \nto\,.
\eeao
Finally, we bound $I_2$ by an application of   \eqref{const:cor:large:deviation:gray}
\beao
 I_2 &\sim& \dfrac{n}{a_n^2} \,c_0 \,\int_{r_n^{1+\delta}}^{(\vep a_n)^2} \P( X >\sqrt{y})\,dy\,,\qquad\nto\,.
\eeao
We also observe that 
\beao
 \int_0^{r_n^{1+\delta}} \P\big(|X|> \sqrt{y}\big)\,dy 
&=&\E\big[ X^2 \1({ |X|}\le r_n^{(1+\delta)/2})\big]+o(1)\,,\qquad \nto\,.
\eeao
The right-hand \fct\ is \slvary\ with argument $r_n^{(1+\delta)/2}$. Hence it is $O(r_n^\gamma)$ for arbitrarily small $\gamma>0$.
Thus by definition of $(a_n)$, 
\beao
 I_2 &=& c_0\,\dfrac{n}{a_n^2} \, \int_0^{(\vep a_n)^2}\P(X>\sqrt{y})\,dy +o(1)\\
&=& c_0\,\dfrac{n}{a_n^2}\, \E[X^2 \1({ |X|} \le \vep a_n)] +o(1)\to c_0\,,\qquad\nto\,. 
\eeao
This finishes the proof.
\end{proof}

\appendix
\section{}
\ble\label{lem:A1}
Assume that the \regvary\ stationary centered \seq\ $(X_t)$ with index $\a=2$ satisfies the relation
\beam\label{eq:9}
\dfrac{X_1^2+\cdots +X_n^2}{a_n^2}\stp \sigma^2\,,\qquad \nto\,,
\eeam
for some \seq\ $a_n=\sqrt{n}\,\ell(n)$, a \slvary\ \fct\ $\ell$ and a constant $\sigma>0$. 
Then the following relation is equivalent to \eqref{eq:9}:
\beam\label{eq:10}
\dfrac{X_1N_1+\cdots +X_n N_n}{a_n}\std N(0,\sigma^2)\,,\qquad \nto\,, 
\eeam
where $(N_t)$ are  iid standard normal \rv s independent of $(X_t)$.
\ele
\begin{proof}
For any $\la\in \bbr$ we have 
\beao
\E\big[\exp\big(-0.5 \la^2\,(X_1^2+\cdots +X_n^2)/a_n^2\big)\big]
&=& \E\big[\exp\big(i\,\la\,(X_1N_1+\cdots +X_nN_n)/a_n\big)\big]\,.
\eeao
Under \eqref{eq:9} (or \eqref{eq:10}) we have \con\ of the left- and right-hand
expressions to\\ $\exp(-0.5(\sigma\,\la)^2)$. Therefore \eqref{eq:9} and \eqref{eq:10}
are equivalent.
\end{proof}
 
\section{An example of an infinite variance iid \seq\ which satisfies the \clt\ but is not \regvary}\label{sec:append:prelim}\setcounter{equation}{0}
We start by recalling  a classical result; see \cite[Section 4.1, Theorem~4.1]{petrov:1995}.
\bth\label{clt:thm:petrov:iid} 
Let $(X_t)$ be a real-valued  iid sequence and $(a_n)$ be real numbers such that $a_n\to\infty$.
Then 
$S_n/a_n\std N(0,\sigma^2)$ for some $\sigma>0$ if and only if
 the following three conditions are satisfied for every $\vep>0$: as $\nto$,
\beam
n \,\P\big(|X / a_n|>\vep \big) &\to& 0\,,\label{eq:6a:append}\\
n\,\var\big((X /a_n)\,\1(|X /a_n|\le \vep) \big)&\to& \sigma^2\,,\label{eq:6b:append}\\
n \,\E\big[(X /a_n)\,\1(|X /a_n|\le \vep) \big]&\to& 0\,.\label{eq:6c:append}
\eeam
\ethe
We will apply this result to the following example.
\bexam\label{exam:1}

We modify the example in Remark 5 of \cite{damek:mikosch:rosinski:samorodnitsky:2014}, cf. also \\ 
\cite{jacobsen:mikosch:rosinski:samorodnitsky:2009}. 
Choose constants $\theta_0\ne 0$,  $a$ and $b$ \st\ $0<a^2+b^2\le 1$. 
Consider a symmetric random variable $X$  with density, for some $r>0$,
\beao
 f(x) = c_r\, |x|^{-3} \,\big(1+a\, \cos (\theta_0 \log|x|)+b\,\sin (\theta_0 \log |x|) \big)\,\1(|x|>r)\,, \qquad c_r^{-1}=\P(|X|>r)\,.
\eeao
First we observe that $\P(X>x)$ is not regularly varying. For $y>r$ we have 
\beam
 \P(X>y) &%= c_r \int_y^\infty x^{-3}\{1+a \cos (\theta_0 \log x)+ b\sin (\theta_0 \log x)\} dx \\
%        = \frac{c_r}{2} \int_0^{y^{-2}}  \{1+a \cos (-\theta_0/2 \log x)+ b\sin (-\theta_0/2 \log x)\} dx \\
         =& \frac{c_r}{2} \int_0^{y^{-2}}  \,\big(1+a \,\cos \big((\theta_0\, \log z)/2\big)- b\,\sin \big((\theta_0 \log z)/2\big)\big) dz\nonumber \\
        &=& \frac{c_r}{2}\Big[
z + C\, z\, \cos \big((\theta_0\, \log z)/2\big) - D\, z \,\sin \big((\theta_0\, \log z)/2\big)
\Big]^{y^{-2}}_0 \label{eq:quanto}\\
         &=&  \frac{c_r}{2} y^{-2}\, \big(1+C\,\cos (\theta_0\, \log y)+D\, \sin (\theta_0 \,\log y)\big)=:    \frac{c_r}{2}\, y^{-2}\,N(y)\,.\nonumber
\eeam
where $
 C:= (a+\theta_0 b/2)/(1+\theta_0^2/4), D:= (b - \theta_0\,a/2)/(1+\theta_0^2/4)$ are obtained by differentiating 
the quantity in brackets in \eqref{eq:quanto} with respect to $z$ and comparing the coefficients of trigonometric functions.
In particular,  
$0 < C^2 + D^2 =(a^2+b^2)/(1+\theta_0^2/4) \le 1$. The \fct\ $N(x)$ is not
\slvary .
 \par
Next we verify the slow variation of $K(x)$. For $x>r$ we have
\beao
K(x)&=& 2\E[X^2\1(0\le X\le x)]= 2\,\Big(\int_{r^2}^{x^2}\P(X>\sqrt{y})dy -x^2\,\P(X>x )\Big) \\
&=& c_r \int_{r^2}^{x^2} y^{-1} N(\sqrt{y})\,dy- c_r \,N(x)\, \\
%&=& 2c_r (\log x- \log r)+ \frac{2c_r}{\theta_0} \big(
%C \sin (\theta_0 \log x)- D \cos (\theta_0 \log x) \big) \\
%&& + \frac{2c_r}{\theta_0} \big(
%D \cos (\theta_0 \log r)- D \sin (\theta_0 \log r)\big) -c_r \,N(x)\, \\
&=&  2\,c_r \log x + c_r\, \big( 2C/\theta_0-D\big) \,\sin (\theta_0 \log x) - c_r\, \big(2D/\theta_0 +C \big) 
\cos (\theta_0 \log x) +c \\
%&=& 2c_r \log x + 2c_r /\theta_0 \big(
%a \sin (\theta_0 \log x) +b \cos(\theta_0 \log x)
%\big) +c \\
&\sim &2\,c_r\,\log x\,,\qquad \xto\,.
\eeao
%\beao
% K(x)&=&4 \,\int_0^x y\,\P(X>y)dy -2\, x^2\P(X>x) \\
%&=& {\blue 4\, \int_0^r y\,dy??}+
% 2 \,c_r \int_r^x y^{-1}\,N(y)\,dy
%\big(1+C\,\cos (\theta_0\, \log y)+D \,\sin (\theta_0 \,\log y)\big)\, dy  \\
%&&\hspace{1.5cm} 
%- c_r\, N(x)\\%\big(1+C\,\cos (\theta_0\, \log x)+D \,\sin (\theta_0\, \log x)\big) \\
%&= &2 \,c_r \,(\log x -\log r) + \dfrac{2\,c_r}{\theta_0}\, \big(C\,\sin (\theta_0\, \log x)-D\, \cos (\theta_0\, \log x)\big) \\
%&& + \dfrac{2\, c_r }{\theta_0} \,\big(D \,\cos (\theta_0 \,\log r) - C\, \sin (\theta_0\, \log r) \big) - c_r\,N(x)\\
%&=&2\, c_r\, \log x +  \dfrac{2\, c_r}{\theta_0}\, \big(C\,\sin (\theta_0 \,\log x)-D \,\cos (\theta_0 \,\log x)\big) \blue -c_r\,N(x)+c\\
%&\sim &2\,c_2\,\log x\,,\qquad \xto\,.
%\eeao
\begin{comment}
where in the $4$th step we use integrals
\begin{align*}
& 2c_r C \int_r^x \frac{\cos (\theta_0 \log y)}{y}dy \\
&= 2c_r C \big[
\theta_0^{-1} \sin (\theta_0 \log y)
\big]_r^x  \\
&= 2 c_r C \theta_0^{-1} \{\sin (\theta_0 \log x) -\sin (\theta_0 \log r) \}
\end{align*}
and 
\begin{align*}
& 2c_r D \int_r^x \frac{\sin (\theta_0 \log y)}{y}dy \\
&= 2c_r D \big[
-\theta_0^{-1} \cos (\theta_0 \log y)
\big]_r^x  \\
&= 2c_r D \theta_0^{-1} \{\cos (\theta_0 \log r) -\cos (\theta_0 \log x) \}. 
\end{align*}
Since the terms other than the first term are bounded, it is easy to observe that 
\[
 \frac{\E[X^2{\bf 1}(|X|\le c x)]}{\E[X^2{\bf 1}(|X|\le x)]}\sim 1\quad \text{as}\quad x \to \infty.
\]
Therefore, $X$ belongs to the domain of normal attraction, cf. Corollary 1 in \cite[VXII.5]{feller:1971}. 
\end{comment}
We choose the normalizing constants $(a_n)$ according to \eqref{eq:3}. Then
it is immediate that $a_n=\sqrt{n}\ell(n)$ for a \slvary\ \fct\ $\ell(n)\to\infty$ 
as $\nto$ and for every $\vep>0$, $n\,\P(|X|>\vep\,a_n)\to 0$ as $\nto$.
Then \eqref{eq:6a:append}, \eqref{eq:6c:append} are satisfied and \eqref{eq:6b:append} holds with 
$\sigma^2=1$. An application of Theorem~\ref{clt:thm:petrov:iid} yields the \clt\
$S_n/a_n\std N(0,1)$ as $\nto$.
\eexam


\begin{thebibliography}{99}
%BBBBBBBBBBBBBBBBBBBBBBBBBBBBBBBBBBBBBBBBBBBBBBBBBBBBBBBBBBBBBBBBBBBB
\bibitem[Basrak and Segers (2009)]{basrak:segers:2009}
{\sc Basrak, B. and Segers, J.}\ (2009) Regularly varying
multivariate time series.
 {\em Stoch. Proc. Appl.} {\bf 119},
1055--1080.
\bibitem[Billingsley (1999)]{billingsley:1999}
{\sc Billingsley, P.}\ (1999)
{\em Convergence of Probability Measures.} 2nd edition.
Wiley, New York.
\bibitem[Bingham et al. (1987)]{bingham:goldie:teugels:1987}
{\sc Bingham, N.H., Goldie, C.M.\ and Teugels, J.L.}\ (1987) {\em
Regular Variation.} Cambridge University Press, Cambridge (UK).
\bibitem[Breidt and Davis (1998)]{breidt:davis:1998}
{\sc Breidt, F.J. and Davis, A.R.}\ (1998)
Extremes of \sv\ models. {\em Ann. Appl. Probab.} {\bf 8}, 664--675.
\bibitem[Buraczewski et al. (2016)]{buraczewski:damek:mikosch:2016}
{\sc Buraczewski, D., Damek, E. and Mikosch, T.}\ (2016)
{\em Stochastic Models with Power-Law Tails. The Equation $X=AX+B$.}
Springer, New York.
\bibitem[Buraczewski et al. (2013)]{buraczewski:damek:mikosch:zienkiewicz:2013}
{\sc Buraczewski, D., Damek. E., Mikosch, T. and Zienkiewicz, J.}\ (2013)  
Large deviations for solutions to stochastic recurrence equations under Kesten’s condition.
{\em Ann. Probab.} {\bf 41} 2755--2790. %https://doi.org/10.1214/12-AOP782
%DDDDDDDDDDDDDDDDDDDDDDDDDDDDDDDDDDDDDDDDDDDDDDDDDDDDDDDDDDDDDDDDDDDDDDD
\bibitem[Damek et al. (2014)]{damek:mikosch:rosinski:samorodnitsky:2014} 
{\sc Damek, E., Mikosch, T. Rosi\'nski, J. and Samorodnitsky, G.}\ (2014) 
General inverse problems for reglar variation. 
{\em J. Appl. Prob.} {\bf 51A}, 229--248.
\bibitem[Davis and Hsing (1995)]{davis:hsing:1995}
{\sc Davis, R.A. and Hsing, T.}\ (1995)
Point process and partial sum \con\ for weakly dependent
\rv s with infinite variance. {\em Ann. Probab.} {\bf 23}, 879--917.
\bibitem[Davis and Resnick (1985a)]{davis:resnick:1985}
{\sc Davis, R.A.\ and Resnick, S.I.}\ (1985)
Limit theory for moving averages of random variables with regularly varying
tail probabilities.
{\em Ann. Probab.} {\bf 13}, 179--195.
\bibitem[Davis and Resnick (1985b)]{davis:resnick:1985a}
{\sc Davis, R.A.\ and Resnick, S.I.}\ (1985)
More limit theory for the sample correlation function of moving averages.
{\em Stoch. Proc.\ Appl.} {\bf 20}, 257--279.
%EEEEEEEEEEEEEEEEEEEEEEEEEEEEEEEEEEEEEEEEEEEEEEEEEEEEEEEEEEEEEEEEEEEEEEEEEEEEEEEEEEEE
\bibitem[Embrechts et al. (1997)]{embrechts:kluppelberg:mikosch:1997}
{\sc Embrechts, P., Kl\"uppelberg, C. and Mikosch, T.}\ (1997)
{\em Modelling Extremal Events for Insurance and Finance.}
Springer,  Berlin.
%FFFFFFFFFFFFFFFFFFFFFFFFFFFFFFFFFFFFFFFFFFFFFFFFFFFFFFFFFFFFFFFFFFFFFFF 
\bibitem[Feller (1971)]{feller:1971}
{\sc Feller, W.}\ (1971) {\em An Introduction to Probability Theory
and Its Applications.} Vol. II. Second edition. Wiley, New York.
%GGGGGGGGGGGGGGGGGGGGGGGGGGGGGGGGGGGGGGGGGGGGGGGGGGGGGGGGGGGGGGGGGGGGGGG 
\bibitem[Gnedenko and Kolmogorov (1954)]{gnedenko:kolmogorov:1954} 
{\sc Gnedenko, B.V. and Kolmogorov, A.N.}\ (1954)
{\em Limit Distributions for Sums of Independent Random Variables.}
Addison-Wesley, Cambridge.
\bibitem[Goldie (1991)]{goldie:1991}
{\sc Goldie, C.M.}\ (1991)
Implicit renewal theory and tails of solutions of random equations.
{\em Ann. Appl.\ Probab.} {\bf 1}, 126--166.

%IIIIIIIIIIIIIIIIIIIIIIIIIIIIIIIIIIIIIIIIIIIIIIIIIIIIIIIIIIIIIIIIIIIIIIIII
\bibitem[Ibragimov and Linnik (1971)]{ibragimov:linnik:1971}
{\sc Ibragimov, I.A. and Linnik, Yu.V.}\ (1971)
{\em Independent and Stationary Sequences of Random Variables.} 
Wolters-Noordhoff Publishing, Groningen.
%JJJJJJJJJJJJJJJJJJJJJJJJJJJJJJJJJJJJJJJJJJJJJJJJJJJJJJJJJJJJJJJJJJJJJJ
\bibitem[Jacobsen et al. (2009)]{jacobsen:mikosch:rosinski:samorodnitsky:2009}
{\sc Jacobsen, M., Mikosch, T., Rosi\'nski, J. and Samorodnitsky, G.}\ (2009) 
Inverse problems for regular variation of linear filters, a cancellation property for $\sigma$-finite measures and identification of stable laws.
{\em Ann. Appl. Probab.} {\bf 19}, 210--242. 
\bibitem[Jakubowski and Szewczak (2020)]{jakubowski:swewczak:2020}
{\sc Jakubowski, A. and Szewczak, Z.S.}\ (2020) Truncated moments of perpetuities
and a new \clt\ for GARCH processes without Kesten's regularity.
{\em Stoch. Proc. Applic.} {\bf 131}, 151--171. 

%KKKKKKKKKKKKKKKKKKKKKKKKKKKKKKKKKKKKKKKKKKKKKKKKKKKKKKKKKKKKKKKKKKKKK
\bibitem[Kesten (1973)]{kesten:1973}
{\sc Kesten, H.\ (1973)}
Random difference equations and renewal theory for
products of random matrices.
{\em Acta Math.} {\bf 131}, 207--248.
\bibitem[Kulik and Soulier (2020)]{kulik:soulier:2020}
{\sc Kulik, R. and Soulier, P.}\ (2020)
{\em Heavy-Tailed Time Series.}  Springer, New York.
%\bibitem{kwapien:woyczynski:1992}
%{\sc Kwapie\'n, S. and Woyczy\'nski}\ (1992)
%{\em Random Series and Stochastic Integrals: Single and Multiple.}
%Birkh\"auser, Boston.
%LLLLLLLLLLLLLLLLLLLLLLLLLLLLLLLLLLLLLLLLLLLLLLLLl
%\bibitem{leadbetter:lindgren:rootzen:1983}
%{\sc Leadbetter, M.R.,\ Lindgren, G.\ and Rootz\'en, H.}\ (1983)
%{\em Extremes and Related Properties of Random Sequences and Processes.}
%Sprin\-ger, Ber\-lin.
\bibitem[Konstantinides and Mikosch (2005)]{konstantinides:mikosch:2005}
{\sc Konstantinides, D.G. and Mikosch, T.}\ (2005)  
Large deviations and ruin probabilities for solutions to 
stochastic recurrence equations with heavy-tailed innovations. 
{\em Ann. Probab.} {\bf 33}, 1992--2035.
\bibitem[Liptser and Shiryaev (1986)]{liptser:shiryaev:1986}
{\sc Liptser, R.Sh. and Shiryaev, A.N.}\ (1986)
{\em Theory of Martingales.} (Russian) Nauka, Moscow.
%MMMMMMMMMMMMMMMMMMMMMMMMMMMMMMMMMMMMMMMMMMMMMMMMMMMMMMMMMMMMMMMMMMMM
\bibitem[Matsui et al. (2024a)]{matsui:mikosch:wintenberger:2024a}
{\sc Matsui, M., Mikosch, T. and Wintenberger, O.} \ (2024)
Self-normalized partial sums of heavy-tailed time series. Technical report.
\bibitem[Matsui et al. (2024b)]{matsui:mikosch:wintenberger:2024b}
{\sc Matsui, M., Mikosch, T. and Wintenberger, O.} \ (2024)
Moments for self-normalized partial sums. Technical report.
\bibitem[Mikosch and Samorodnitsky (2000)]{mikosch:samorodnitsky:2000}
{\sc Mikosch, T. and Samorodnitsky, G.} \ (2000)
The supremum of a negative drift random walk with 
dependent heavy-tailed steps. {\em Ann. Appl. Probab.} {\bf 10},
1025--1064.
\bibitem[Mikosch and St\u aric\u a (2000)]{mikosch:starica:2000}
{\sc Mikosch, T. and St\u aric\u a, C.}\ (2000)
Limit theory for the sample autocorrelations and extremes
of a GARCH(1,1) process.
{\em Ann. Statist.} {\bf 28},
1427--1451.
\bibitem[Mikosch and Wintenberger (2013)]{mikosch:wintenberger:2013}
{\sc Mikosch, T. and Wintenberger, O.}\ (2013)
Precise large deviations for dependent regularly varying sequences.
{\em Probab. Th. Rel. Fields.} {\bf 156},
851--887.
\bibitem[Mikosch and Wintenberger (2014)]{mikosch:wintenberger:2014}
{\sc Mikosch, T. and Wintenberger, O.}\ (2014)
The cluster index of \regvary\ \seq s with applications to limit theory
for \fct s of multivariate Markov chains.~{\em Probab. Th. Rel. Fields.} {\bf 159}, 157--196.
\bibitem[Mikosch and Wintenberger (2016)]{mikosch:wintenberger:2016}
{\sc Mikosch, T. and Wintenberger, O.}\ (2016)
A large deviations approach to limit theory for heavy-tailed
time series.~{\em Probab. Th. Rel. Fields.} {\bf 166}, 233--269.
\bibitem[Mikosch and Wintenberger (2024)]{mikosch:wintenberger:2024}
{\sc Mikosch, T. and Wintenberger, O.}\ (2024)
{\em Extreme Value Theory for Time Series. Models with Power-Law Tails.}
Springer, New York.
\bibitem[Mokkadem (1990)]{mokkadem:1990}
{\sc Mokkadem, A.}\ (1990) Propri\'et\'es de m\'elange des processus
autor\'egressifs polynomiaux. {\em Ann. Inst. H. Poincar\'e,
  Probab. Statist.} {\bf 26}, 219--260.
%NNNNNNNNNNNNNNNNNNNNNNNNNNNNNNNNNNNNNNNNNNNNNNNNNNNNNNNNNNNNNNNNNNN
\bibitem[Nagaev (1969a)]{nagaev:1969a}
{\sc Nagaev, A.V.}\ (1969) 
Limit theorems for large deviations where Cram\'er's conditions are violated 
(in Russian). 
{\em Izv.\ Akad.\ Nauk UzSSR, Ser.\ Fiz.--Mat.\ Nauk} {\bf 6}, 17--22.
\bibitem[Nagaev (1969b)]{nagaev:1969}
{\sc Nagaev, A.V.}\ (1969)
Integral limit theorems for large deviations when Cram\'er's condition
is not fulfilled~I,II.
{\em Th. Probab. Appl.} {\bf 14}, 51--64 and 193--208.
\bibitem[Nagaev (1977)]{nagaev:1977}
{\sc Nagaev, A.V.}\ (1977)
A property of sums of independent random variables.
{\em Th. Probab. Appl.} {\bf 22}, 335--346.
\bibitem[Nagaev (1979)]{nagaev:1979}
{\sc Nagaev, S.V.}\ (1979)
Large deviations of sums of independent random
variables.
{\em Ann. Probab.} {\bf 7}, 745--789.
%PPPPPPPPPPPPPPPPPPPPPPPPPPPPPPPPPPPPPPPPPPPPPPPPPPPPPPPPPPPPPPPPPPPP
\bibitem[Peligrad and Sang (2013)]{peligrad:2013}
{\sc Peligrad, M. and Sang, H.}\ (2013)
Central limit theorem for linear processes with infinite variance.
{\em J. Theor. Probab.} {\bf 26}, 222--239.
\bibitem[Peligrad et al. (2022)]{peligrad:2022}
{\sc Peligrad, M., Sang, H., Xiao, Y. and Yang, G.}\ (2022)
Limit theorems for linear random fields with innovations in the domain
of attraction of a stable law. {\em Stoch. Proc. Appl.} {\bf 150}, 596--621. 
\bibitem[Petrov (1995)]{petrov:1995}
{\sc Petrov,~V.V.} \ (1995) {\em Limit Theorems of Probability
Theory. Sequences of Independent Random Variables.}
Oxford Studies in Probability  {\bf 4.}
Oxford University Press, New York.
\bibitem[Philipps and Solo (1992)]{philipps:solo:1992}
{\sc Philipps. P.C.B. and Solo, V.}\ (1992)
Asymptotics for linear processes. {\em Ann. Statist.}
{\bf 20}, 971--1001.
\bibitem[Pruitt (1981)]{pruitt:1981}
{\sc Pruitt, W.E.}\ (1981) 
The growth of random walks and L\'evy processes.
{\em Ann. Probab.} {\bf 9}, 948--956.
%WWWWWWWWWWWWWWWWWWWWWWWWWWWWWWWWWWWW
\bibitem[Williams (1991)]{williams:1991}
{\sc Williams, D.}\ (1991)
{\em Probability with Martingales.} Cambridge University Press, Cambridge (UK).
\end{thebibliography}
\end{document}